\documentclass{article}
\usepackage{graphicx,geometry,amsmath,amsthm,amssymb,mathdots,authblk,hyperref}
\usepackage{array,dsfont}

\theoremstyle{plain}
\newtheorem{theorem}{Theorem}[section]
\newtheorem{lemma}[theorem]{Lemma}
\newtheorem{prop}[theorem]{Proposition}
\newtheorem{corollary}[theorem]{Corollary}
\newtheorem*{corollary*}{Corollary}

\theoremstyle{definition}
\newtheorem{remark}[theorem]{Remark}

\theoremstyle{definition}

\makeatletter
\newcommand{\smallpmod}[1]{\mkern 8mu({\operator@font mod}\mkern 6mu#1)}
\makeatother

\numberwithin{equation}{section}

\title{Green's function on the Tate curve}

\author{An Huang, Rebecca Rohrlich, Yaojia Sun, Eric Whyman}
\date{}

\begin{document}

\maketitle

\begin{abstract}
Motivated by the question of defining a $p$-adic string worldsheet action in genus one, we define a Laplacian operator on the Tate curve, and study its Green's function. We show that the Green's function exists. We provide an explicit formula for the Green's function, which turns out to be a non-Archimedean counterpart of the Archimedean Green's function on a flat torus. In particular, it turns out that this Green's function recovers the N\'eron local height function for the Tate curve in the $p\to\infty$ limit, when the $j$-invariant has odd valuation. So this non-Archimedean height function now acquires a physics meaning in terms of the large $p$ limit of a non-Archimedean conformal field theory two point function on the Tate curve, as well as a direct analytic interpretation as a Green's function, on the same footing as in the Archimedean place.
\end{abstract}

\section{Introduction}

In 1989, Freund and Witten \cite{FW1987} observed that the Veneziano amplitude satisfies an infinite product formula. Together with \cite{FO1987}, they proposed that there exists a $p$-adic open string worldsheet action in genus zero, which should produce a $p$-adic Veneziano amplitude, such that the reciprocal of the product of all these $p$-adic amplitudes recovers the Archimedean Veneziano amplitude. They also extended the observation to closed strings. Soon after, Zabrodin \cite{Z1989} confirmed this proposal, by showing that a $p$-adic open string action on the Bruhat-Tits tree $T_p$ of the group $PGL(2,\mathbb{Q}_p)$ does exactly that. By integrating out bulk fields, Zabrodin further confirmed that this action has an equivalent description (or called dual description) by another action defined on $\mathbb{Q}_p$, i.e. the asymptotic boundary of $T_p$ minus a point at infinity.

This equivalent action on $\mathbb{Q}_p$ turns out to be an action given by a single kinetic term defined by the regularized Vladimirov derivative, which is known to be a pseudo-differential operator whose symbol is given by the $p$-adic norm function. Deforming the norm function by a quasi-character of $\mathbb{Q}_p^{\times}$, one obtains a family of deformations of the regularized Vladimirov derivative, and therefore also this free action \cite{HSZ2022}. It turns out these are generalized free actions. Therefore the Green's function of this family of operators is a crucial object to compute. Here, a striking observation is that these Green's functions are given by the local functional equation in Tate's thesis. Furthermore, the global functional equation in Tate's thesis turns out to be equivalent to a product formula of these Green's functions: i.e., upon a regularization by analytic continuation, the product of these Green's functions over all places is equal to 1.

Later, it was also observed that this deformed family of action defined at all the places of a number field is of interest from several points of view \cite{HSZ2022}: it is closely related to dimensional regularization in physics and Hecke L-functions in number theory. At the complex place, it turns out that the deformation gives rise to a way to construct Verma modules of $sl(2,\mathbb{C})$. It gives rise to a family of conformal field theories at non-Archimedean places, in the sense that the action is invariant under a deformed action of a central extension of $GL(2,F)$ on the scalar field, where $F$ is the non-Archimedean local field. This representation theoretic aspect and its generalization shall be investigated in a separate article. Furthermore, it offers a new physical interpretation of the quadratic reciprocity law, that can potentially generalize to also incorporate more general abelian reciprocity laws \cite{HSZ2022}, as the family of actions enjoy both global conformal symmetry and Galois group symmetry, which commute with each other. In addition, very recently it has been realized that this deformation family also appears to give rise to a $p$-adic counterpart of the classical work of Caffarelli and Silvestre \cite{CL2007} regarding fractional Laplacians, that at the same time also generalizes Zabrodin's dual pair of theories on the Bruhat-Tits tree $T_p$ and on its boundary \cite{Z1989} to the deformations. This aspect shall be investigated in a separate article.

As a next step, one is interested in looking for a $p$-adic string action in genus one. This has been an important open question. As the genus zero story shows, one would expect the genus one theory to be important for both physical and mathematical reasons.

In this article, we consider a special case given by the Tate curve, and propose a $p$-adic string worldsheet action on a Tate curve. We take ($\mathbb{Q}_p$-points of) the Tate curve $E_q=\mathbb{Q}_p^{\times}/q^{\mathbb{Z}}$ with $|q|<1$. A real valued locally constant scalar field on $E_q$ can be pulled back to a periodic function $\phi: \mathbb{Q}_p^{\times}\to\mathbb{R}$. 

We propose the following worldsheet action for the Tate curve 
\begin{equation}\label{action}
S=c\int_{E}\phi(x)D\phi(x)\,d\mu^{\times},
\end{equation}
where $d\mu^{\times}=\frac{dx}{|x|}$ is a multiplicative Haar measure, $c$ is a normalization constant, and $E$ is a fundamental domain of $E_q$.
\begin{equation}\label{D}
D\phi(x):=\int_{E}\frac{\phi(z)-\phi(x)}{|z-x|^2}|x|\,dz,
\end{equation}
where $dz$ is an additive Haar measure on $\mathbb{Q}_p$. 

The main idea here is that the action is locally identical to the genus zero $p$-adic Bosonic string action defined on the asymptotic boundary of $T_p$ by the regularized Vladimirov derivative, a non-Archimedean conformal field theory of free scalar Bosons mentioned earlier. Also see e.g. \cite{HMST}-- the only difference is the domain of integration is now $E$ instead of $\mathbb{Q}_p$. Recall that in Archimedean string theory, the worldsheet actions of any genus are locally identical after gauge fixing the worldsheet metric, as a local patch of a real surface is conformally flat.

Therefore, we propose $D$ to be a $p$-adic version of the flat Laplacian on the torus. 

\begin{remark}

In genus 0, the regularized Vladimirov derivative $D_1$ mentioned above is defined by the following integral operator on the space of compactly supported locally constant functions on $\mathbb{Q}_p$:
\begin{equation}
D_1\phi(x):=\int_{\mathbb{Q}_p}\frac{\phi(z)-\phi(x)}{|z-x|^2}\,dz,
\end{equation}

with the free scalar Boson CFT action given by
\begin{equation}
S_1=c\int_{\mathbb{Q}_p}\phi(x)D_1\phi(x)\,dx,
\end{equation}

As mentioned earlier, $D_1$ is a pseudo-differential operator whose symbol is the $p$-adic norm function, thus it is a direct $p$-adic counterpart of the 2d Laplacian on the complex plane. This has been generalized to pseudo-differential operators defined by an arbitrary quasi-character on a characteristic zero local field in \cite{huang2021greens}. Again as mentioned earlier, this family of operators turns out to be important for a variety of reasons. Non-Archimedean pseudo-differential operators in a more general setup have been considered in e.g. \cite{MR3586737}.

Consider the Archimedean exponential map $w\to z=x+yi=e^{-iw}$, mapping the cylinder to $\mathbb{C}^{\times}$, or half the cylinder to the upper half plane. $w=\sigma+i\tau$ where $\sigma, \tau$ are spatial and time coordinates on the worldsheet of a free closed or open string. Under the coordinate transformation from $\sigma,\tau$ to $x,y$, the Laplacian gets multiplied by the number theoretic norm $|x+yi|_\mathbb{C}=x^2+y^2$, and the measure $d\sigma \,d\tau$ becomes $\frac{1}{|x+yi|_\mathbb{C}}\,dx\,dy$. This analogy provides another motivation for our definition of the action $S$ and the operator $D$.
\end{remark}

Since multiplication by a unit does not change the above action, we consider only the case $q=p^m$. Take $E$ to be $\cup_{s=0}^{m-1}p^s\mathbb{Z}_p^\times$. One checks that the action is independent of the choice of $E$: i.e., replacing $E$ by a shift of $E$ by a power of $p^m$, the action is invariant.

To study the physics, it is crucial to understand the operator $D$. One sees readily that some of the basic properties of $D$ are similar to the genus zero case, that $D$ is self-adjoint, negative semi-definite, and preserves locally constant functions under the integral pairing $<f, g>:=\int_Efg\,d\mu^{\times}$.

A technical advantage here is that one can filter the space of fields by finite-dimensional vector spaces of functions on $\mathbb{Q}_p^{\times}$ descending to $E_q$, such that the conductor of these functions is bounded by a given natural number $k$. Equivalently, one considers projections to the finite quotients $\pi^k: E_q\to E_q/(1+p^k\mathbb{Z}_p)$, and considers only functions that are pullbacks of functions from these quotients. This is a filtration by functions depending on only up to the leading $k$ $p$-adic digits. 

Alternatively, this filtration is nothing but truncating the graph $T_p/\Gamma$ up to radius $k$, where $\Gamma$ is the discrete subgroup of $PGL(2,\mathbb{Q}_p)$ generated by $\begin{bmatrix} q & 0\\0 & 1\end{bmatrix}$. Note that the same type of truncation on the Bruhat-Tits tree $T_p$ was used in a crucial way by Zabrodin to compute the $p$-adic genus zero 4-tachyon scattering amplitude, and to derive the asymptotic boundary dual $p$-adic string worldsheet action.

The Green's function is again of central importance, as it is the two point correlation function of the action \eqref{action}, up to a shift by a constant. It is also expected to be a basic building block for computing $p$-adic string torus amplitudes.
For each $k$, the Green's function is simply a solution to a finite dimensional matrix equation $DG=\delta-\frac{1}{V}$, where $\delta$ is the Dirac Delta distribution, and $V=m\mu^{\times}(\mathbb{Z}_p^\times)$ is the volume of $E_q$.

As is mentioned earlier, for the $p$-adic version of the flat Laplacian and its fractional powers in the genus zero case, defined on $\mathbb{Q}_p$, their Green's functions turn out to be closely related to Tate's thesis \cite{huang2021greens}. On the other hand, the Green's function on the flat Archimedean torus has been a classical object of interest. In particular, in a recent breakthrough, Lin and Wang described its number of critical points as a function of the torus moduli \cite{lin2010elliptic,lin2010function}. Therefore, from these perspectives, it would also be of interest to investigate the $p$-adic counterpart of the Laplacian on the flat torus, and its Green's function. There are results for the Green's function in somewhat different contexts, e.g. on the upper half-plane and p-adic domains \cite{hassan2025padichighergreensfunctions,bradley2025boundaryvalueproblemspadic}, and the existence of the Green's function on p-adic manifolds with respect to the fractional Laplacian for $s>1$ has been proved \cite{bradley2025diffusionoperatorspadicanalytic}.

The main result of this paper is an explicit formula for this Green's function, in terms of a power series plus a simple correction term:
\begin{theorem}\label{MainThm}
    The unique (up to an additive constant) symmetric Green's function for the $p$-adic Laplacian $D$ on the Tate curve is given by $G_{p,m}(x,y)=B_{p,m}(x,y)+C_{p,m}(x,y)$, where
    \begin{equation*}\tag{\ref{E:Bpmxy}}
        B_{p,m}(x,y)=A(y)+\lambda_0\log(d(x,y))+\sum_{n=1}^{\infty}\lambda_n(y)\,d(x,y)^n,
    \end{equation*}
    $d(x,y)=\dfrac{|x-y|}{\max\{|x|,|y|\}}$, and $C_{p,m}(x,y)$ depends only on $|x|$ and $|y|$ and is given by the recursive formula \eqref{E:CRecur}. The constant $\lambda_0\in\mathbb{R}$ is given by \eqref{E:lambda0}, and $\lambda_n(y),A(y)\in\mathbb{R}$ depend only on $|y|$ and are given by \eqref{E:lambdan} and \eqref{E:Ay}, respectively.
\end{theorem}

This is a $p$-adic counterpart of the Archimedean case. In particular, in both cases, the Green's function near the diagonal is given by a log singularity, plus a power series correction. In the Archimedean case, the power series correction is needed because it makes the Green's function well-defined on the torus, whereas in the non-Archimedean case, it turns out that the power series correction is needed because the log singularity itself does not satisfy the Green's equation. On the other hand, there is another very recently discovered closed formula for this Green's function in terms of a finite sum, and a third closed formula in terms of the q-digamma function for $q=p$, both of which we plan to investigate in the near future. 

In addition, one can analyze the spectrum of $D$. We have observed that the spectral gap behaves as expected, and the spectrum does satisfy a Weyl asymptotics, that is a non-Archimedean counterpart of the Archimedean version of the Weyl asymptotics. Furthermore, we have studied the Bosonic partition function given by the spectrum of $D$ when $m=1$, and found that its leading term gives a $p$-adic counterpart of the entropy of the Archimedean Bosonic string. Some of these issues are investigated in detail in \cite{HJ2025}.

Last but not least, it turns out that our Green's function recovers the N\'eron local height function for the Tate curve in the $p\to\infty$ limit, when the valuation of the $j$-invariant is odd:
\begin{corollary*}[Corollary \ref{HeightCor}]
When $m$ is odd, after choosing the new ``central normalization'' of $C_{p,m}(x,y)=0$ for $v_{p}(x)=v_{p}(y)=\frac{m-1}{2}$, we have 
\begin{equation*}\tag{\ref{height}}
\lim_{p\to\infty}\left(-\frac{1}{p}G\left(xp^{\frac{m-1}{2}},p^{\frac{m-1}{2}}\right)\right)=h(x)-\frac{m}{12},
\end{equation*} where $h(x)$ is the N\'eron local height of the point $x$ \cite{Silverman}.
\end{corollary*}
A consequence is that conceptually, one can now think of this non-Archimedean local height function in terms of a Green's function on the curve, instead of having to resort to, e.g., the Berkovich space at non-Archimedean places. Speaking in terms of physics, this translates into the observation that the local height function on the Tate curve is the large $p$ limit, or the leading term of the two point correlation function of the free scalar conformal field theory on the Tate curve. So this important arithmetic geometric function aquires  direct physics and analytic meaning! 

\section*{Acknowledgement}
AH and YJS are grateful to SIMIS for its hospitality during an important phase of this work. The work of AH is supported by Simons collaboration grant No. 708790. AH thanks Chin-Lung Wang for first pointing out the possibility that the Green's function discussed in this paper could be related to the height function. The work of EW is partially supported by the Robert and Charlotte Joly Endowed Scholarship Fund.

\subsection{Outline}

\hspace{\parindent}In section 2, we investigate some basic properties of our $p$-adic Laplacian $D$ and the Green’s function $G$. We then prove the existence of the Green’s function by taking the limit of the solutions on finite quotients.

In section 3, we prove that the Green's function has the form $G(x,y)=B(x,y)+C(x,y)$, where $B$ (given by \eqref{E:Bpmxy}) is an infinite series in $d(x,y)=\dfrac{|x-y|}{\max\{|x|,|y|\}}$, depending only on $|x|$, $|y|$, and $|x-y|$; $C$ depends only on $|x|$ and $|y|$; and both $B$ and $C$ are symmetric. In the $m=1$ case, $G=B$; $C$ is a `correction' to $B$ to account for the dependence of $DB(x,y)$ on $|x|$ and $|y|$. We compute an explicit infinite convergent series expansion of $B(x,y)$ in $d(x,y)$ by correcting a log singularity on the diagonal.

In section 4, we study the properties of $C(x,y)$, including proving its existence (which completes the proof that $G(x,y)=B(x,y)+C(x,y)$), and we prove a set of recursive formulas for explicitly computing $C$, which completes the proof of the Main Theorem \eqref{MainThm}.

In section 5, we explain the relation between our Green's function and the N\'eron local height function for the Tate curve. We also give remarks on a potential framework from physics, to understand the relation in a more conceptual way. 

In section 6, we generalize our results from sections 3 and 4 to finite extensions of $\mathbb{Q}_p$.

\subsection{Table of notation}

\begingroup
\begin{center}
\renewcommand{\arraystretch}{1.5}
\begin{tabular}{|c|l|}
\hline
\textbf{Notation} & \textbf{Meaning} \\\hline\hline
$p$ & A fixed prime $p\in\mathbb{N}$, considered as either an element of $\mathbb{R}$ or of $\mathbb{Q}_p$ \\\hline
$v_p$ & The $p$-adic valuation on $\mathbb{Q}_p$ \\\hline
$|\cdot|$ & The $p$-adic norm on $\mathbb{Q}_p$; $|x|=p^{-v_p(x)}$ \\\hline
$\log$ & The logarithm base $p$ \\\hline
\rule{0pt}{17pt}$d(x,y)$ & $\dfrac{|x-y|}{\max\{|x|,|y|\}}$ \\[7pt]\hline
\rule{0pt}{20pt}$E$ & A fundamental domain of the Tate curve $E_q=\mathbb{Q}_p^\times/q^\mathbb{Z}$; for $q=p^m$, $E=\displaystyle\bigcup_{s=0}^{m-1}p^s\mathbb{Z}_p^\times$ \\[10pt]\hline
\rule{0pt}{17pt}$D$ & $D\phi(x):=\displaystyle\int_{E}\dfrac{\phi(z)-\phi(x)}{|z-x|^2}|x|\,dz$ \\[7pt]\hline
$G$ & The Green's function for $D$; $G=B+C$ (for $q=p^m$, we denote $G$ by $G_{p,m}$) \\\hline
$B$ & The part of $G$ given as a series in $d(x,y)$ (denoted $B_{p,m}$ for $q=p^m$) \\\hline
$C$ & The `correction' to $B$ depending only on $|x|$ and $|y|$ (denoted $C_{p,m}$ for $q=p^m$) \\\hline
\rule{0pt}{21pt}$\partial y_\ell$ & $\left\{z\in\displaystyle\bigcup_{s=0}^{m-1}p^s\mathbb{Z}_p^\times\,\middle|\,v_p(z)=v_p(y),v_p(z-y)=\ell\right\}$ \\[11pt]\hline
$U(x)$ & $p^{-v_p(x)}+p^{-m+v_p(x)+1}=|x|+p^{1-m}|x|^{-1}$ (sometimes denoted $U(i)$ for $i=v_p(x)$) \\\hline
$\Lambda_i$ & $1+p^{-1}-p^{-i-1}-p^{-m+i}=1+p^{-1}-p^{-1}U(i)$ \\\hline
$\mu^+$ & The Haar measure on $\mathbb{Q}_p$ with $d\mu^{+}=dx$, normalized to have $\mu^+(\mathbb{Z}_p)=1$ \\\hline
\rule{0pt}{12pt}$\mu^\times$ & The Haar measure on $\mathbb{Q}_p^{\times}$ with $d\mu^{\times}=\frac{dx}{|x|}$, normalized to have $\mu^\times(\mathbb{Z}_p^\times)=1-p^{-1}$ \\[2pt]\hline
\end{tabular}
\end{center}
\endgroup

\newpage

\section{Basic properties of \texorpdfstring{$D$}{D} and \texorpdfstring{$G$}{G}}

\begin{lemma} \label{DPresLCLemma}
    $D$ is self-adjoint, negative semi-definite, and preserves locally constant functions under the integral pairing $<f, g>\,:=\int_Efg\,d\mu^{\times}$.
\end{lemma}

\begin{proof}
    We prove it by direct calculation. First, for any locally constant function $f$ and $g$ on $E$, we have
    \begin{align*}
        <Df,g>\,=&\int_EDf(x)g(x)\,d\mu^{\times}(x)\\
        =&\int_E\left(|x|\int_E\frac{f(z)-f(x)}{|z-x|^2}\,dz\right)g(x)\,\frac{dx}{|x|}\\
        =&\int_E\int_E\frac{f(z)g(x)-f(x)g(x)}{|z-x|^2}\,dz\,dx\\
        =&\int_E\int_E\frac{g(z)f(x)-f(x)g(x)}{|z-x|^2}\,dz\,dx\\
        =&<f,Dg>
    \end{align*}
    thus $D$ is self adjoint.
    
    Next, as
    \begin{align*}
        <Df,f>\,=&\int_EDf(x)f(x)\,d\mu^{\times}(x)\\
        =&\int_E\left(|x|\int_E\frac{f(z)-f(x)}{|z-x|^2}\,dz\right)f(x)\,\frac{dx}{|x|}\\
        =&\int_E\int_E\frac{f(z)f(x)-f(x)^2}{|z-x|^2}\,dz\,dx\\
        =&\int_E\int_E\frac{f(x)f(z)-f(z)^2}{|z-x|^2}\,dz\,dx\\
        =&\,\frac{1}{2}\int_E\int_E\frac{2f(z)f(x)-f(x)^2-f(z)^2}{|z-x|^2}\,dz\,dx\\
        =&-\frac{1}{2}\int_E\int_E\frac{(f(z)-f(x))^2}{|z-x|^2}\,dz\,dx\\
        \leq&\,0
    \end{align*}
    and we notice that $D\mathds{1}_E(x)=0$, so $D$ is negative semi-definite.
    
    Last, as $f$ is locally constant, then $\forall x\in E$, $\exists$ a neighborhood $x+p^n\mathbb{Z}_p$ of $x$ for some $n$, such that $\forall x_0\in p^n\mathbb{Z}_p$, we have $f(x)=f(x+x_0)$. Then
    \begin{align*}
        Df(x)=|x|\int_E\frac{f(z)-f(x)}{|z-x|^2}\,dz=|x+x_0|\int_E\frac{f(z)-f(x+x_0)}{|z-(x+x_0)|^2}\,dz=Df(x+x_0)
    \end{align*}
    which means $D$ preserves locally constant functions.
\end{proof}

\begin{lemma}
For each $k$, there is a unique symmetric Green's function on the finite quotient of the Tate curve $E_q/(1+p^k\mathbb{Z}_p)$, up to adding a constant.
\end{lemma}

\begin{proof}
First, one verifies from the integral definition of $D$, that $D$ preserves the conductor $k$. Therefore, the Green's equation is a well-defined equation on each such finite quotient $E_k$. One also checks that the kernel of $D$ consists of constant functions. Furthermore, the RHS of the Green's equation also has kernel given by constant functions.

Up to an overall scaling, delta functions supported at points of $E_k$ give rise to an orthonormal basis w.r.t. the integral pairing as in Lemma \ref{DPresLCLemma}. The Green's equation on $E_k$, as a matrix equation, may be viewed as a matrix presentation of an equation of linear transformations, under this basis. 

As $D$ is self-adjoint, on $E_k$, $D$ projects the function space $V$ onto $W$, the orthogonal complement of its kernel. Choose an orthonormal basis of $W$, then its image under the composition $W\to V\to V/<1>$ also gives rise to a basis of the quotient space $V/<1>$. Consider the Green's equation on this quotient space, under this basis, $D$ is symmetric and nonsingular. The RHS becomes the identity matrix. Thus it has a unique solution, and the solution is symmetric. Next, by the Green's equation, any Green's function on $E_k$ has to map the constant function to the kernel of $D$, i.e. the constant function. So the Green's function preserves constant functions. This implies that upon extending our orthonormal basis of $W$ to an orthonormal basis of $V$ by adding a constant function, our unique solution to the Green's equation on $V/<1>$ extends to a Green's function on $V$, represented by a block upper triangular matrix under our orthonormal basis, with the 1st block of size 1, corresponding to the subspace $<1>$. Therefore, one may subtract suitable constants from each column, to get a Green's function that is represented by a symmetric matrix under this basis. So the resulting Green's function is also symmetric under the old delta function basis, as these two basis are related by an orthogonal transformation.

Finally, if there are two symmetric Green's functions on $E_k$, their difference is again symmetric, and is annihilated by $D$. Therefore the difference is a constant matrix.
\end{proof}

\begin{corollary}
If there is a symmetric Green's function on the Tate curve, as a distribution on the space of continuous functions on the Tate curve, then it is the unique symmetric Green's function on the Tate curve up to adding a constant.
\end{corollary}
\begin{proof}
Going to the finite quotients, together with the fact that functions that descend to finite quotients are dense in the space of continuous functions on the Tate curve, the corollary is obtained.
\end{proof}

\begin{corollary}
Any symmetric Green's function on any $E_k$ is invariant under multiplication by units: i.e. $G(ux,uy)=G(x,y)$ for any $p$-adic unit $u$.
\end{corollary}
\begin{proof}
Let $G_1(x,y):=G(ux,uy)$. Then
\begin{align*}
  DG_1(x,y)&=\int_E\frac{G_1(z,y)-G_1(x,y)}{|z-x|^2}|x|\,dz\\
&=\int_E \frac{G(uz,uy)-G_1(ux,uy)}{|z-x|^2}|x|\,dz\\
&=\int_E \frac{G(uz,uy)-G_1(ux,uy)}{|uz-ux|^2}|ux|\,duz\\
&=\delta(ux,uy)-\frac{1}{V}\\
&=\delta(x,y)-\frac{1}{V}\\  
\end{align*}

Thus $G_1(x,y)$ is a symmetric Green's function on $E_k$. Since $E_k$ is a finite set, by the uniqueness of symmetric Green's function, $G_1(x,y)=G(x,y)$. (There is no ambiguity here of adding a constant)
\end{proof}

\begin{lemma}\label{digits}
For any fixed $y$, the Green's function $G(x,y)$ on any $E_k$ depends on $x-y$ only up to its leading $(m-v_p(y))$-th $p$-adic digit, where $m=v_p(q)$. 
\end{lemma}
\begin{proof}
    Let $x'\in\cup_{i=0}^{m-1}p^i\mathbb{Z}_p^\times$ such that $x'$ and $x$ share the leading $(m-v_p(y))$ many $p$-adic digits. i.e. there exists a unit $u$, such that $x'=y+u(x-y)=ux+(1-u)y$. Note that $|x'|=|x|$, and $(1-u)y\in p^m\mathbb{Z}_p$, so any of the $p^i\mathbb{Z}_p^\times$ is invariant under the translation by $(1-u)y$. Denote $G_2(x,y)=G(x',y)$. Let $z'=uz+(1-u)y$, then we have
    \begin{align*}
    DG_2(x,y)&=\int_E \frac{G_2(z,y)-G_2(x,y)}{|z-x|^2}|x|\,dz\\
    &=\int_E \frac{G(z',y)-G(x',y)}{|z-x|^2}|x|\,dz\\
    &=\int_E \frac{G(z',y)-G(x',y)}{|z'-x'|^2}|x'|\,dz'\\
    &=\delta(x',y)-\frac{1}{V}\\
    &=\delta(x',y)-\frac{1}{V}
    \end{align*}
    Thus $G(x',y)$ is a Green's function for our fixed $y$, thus as a function of $x$, it differs from $G(x,y)$ by adding a constant, which obviously has to be zero.
\end{proof}

Consider the projection: $\pi_k: S/(1+p^{k+1}\mathbb{Z}_p)\to S/(1+p^{k}\mathbb{Z}_p)$. 

Let $G_k$ denote a Green's function with respect to $k$. Define integration along the fiber
\begin{equation*}
    \tilde{G}_k(x,y):= \frac{1}{p^2}\sum_{x': \pi_k(x')=x, y': \pi_k(y')=y} G_{k+1}(x',y')
\end{equation*}

\begin{lemma}
   $\tilde{G}_k(x,y)$ is a Green's function for $E_k$. 
\end{lemma} 
\begin{proof}
Let $\phi$ be a test function on $E_k$. Then 
\begin{align*}
    \int_E \phi(y) D\tilde{G}_k(x,y)\,dy
    &=\frac{1}{p^2}\sum_{x': \pi_k(x')=x, y': \pi_k(y')=y}\int_E \phi(y) DG_{k+1}(x',y')\,dy\\
    &=\frac{1}{p^2}\sum_{x': \pi_k(x')=x, y': \pi_k(y')=y}\int_E \delta(x',y')\phi(y)p\,dy'-\frac{1}{V}\phi(y)\\
    &=\frac{1}{p^2}\sum_{x': \pi_k(x')=x}p\phi(x')-\frac{1}{V}\phi(y)\\
    &=\phi(x)-\frac{1}{V}\phi(y)\\
    &=(\delta(x,y)-\frac{1}{V})\phi(y)
\end{align*}
Since for any fixed lift $y'$ of $y$, $dy=pdy'$, and for any lift $x'$ of $x$, there exists exactly one lift $y'$ of $y$, whose $(k+1)$-th digit matches with that of $x'$. 

\end{proof}

Therefore, integration along fiber maps a symmetric Green's function to a symmetric Green's function. 

Next, observe that all the Green's equations on finite quotients $E_k$ are matrix equations over $\mathbb{Q}$. In particular, there are real symmetric solutions. We further uniformize these real symmetric solutions by requiring that the maximum value of any of the solutions is $0$. This fixes a unique solution for each $k$, which by abuse of notation, we still denote by $G_k$.

Combined with Lemma \ref{digits}, we deduce that for any $x\neq y$ points of $E$, when $k$ is big enough, 
\begin{equation}
    \tilde{G}_k(x,y)={G}_{k+1}(x,y)
\end{equation}

In the above equation, we again used abuse of notation, where arguments on both sides are really given by projections from $E$.

So $\tilde{G}_k(x,y)$ is a real symmetric Green's function on $E_k$ with maximum value $0$, thus $\tilde{G}_k(x,y)=G_k(x,y)$. As a consequence, the limit of $G_k(x,y)$ exists for any $x\neq y$.

Denote this limit by $G(x,y)$. We next show that $G(x,y)$ is a real symmetric Green's function on the Tate curve:

\begin{lemma}
The above constructed $G(x,y)$ is a real symmetric Green's function on the Tate curve. I.e., as a distribution on the space of continuous functions on $E$, $DG(x,y)=\delta(x,y)-\frac{1}{V}$.
\end{lemma}
\begin{proof}
Again, functions descending to a finite quotient $E_k$ are dense in the space of continuous functions on the compact $E$. So it suffices to check the Green's equation for each $E_k$, which follows from the definition of $G(x,y)$.
\end{proof}

We therefore have proved
\begin{theorem}\label{uniqueness}
There is a unique symmetric Green's function on the Tate curve, up to the addition of a constant.
\end{theorem}

The Green's function on the Tate curve has the following symmetry, which is a $p$-adic counterpart of the obvious reflection symmetry in the Archimedean case:
\begin{prop}\label{double flip}
$G(x,y)=G(p^{m-1}x^{-1},p^{m-1}y^{-1})$
\end{prop}
\begin{proof}
Denote $G_1(x,y)=G(p^{m-1}x^{-1},p^{m-1}y^{-1})$. We show that $G_1(x,y)$ is a symmetric Green's function:
Substituting in $z_1=p^{m-1}z^{-1}$, $x_1=p^{m-1}x^{-1}$, and $y_1=p^{m-1}y^{-1}$, we have

\begin{align*}
DG_1(x,y)&=\int_E \frac{G(p^{m-1}z^{-1},p^{m-1}y^{-1})-G(p^{m-1}x^{-1},p^{m-1}y^{-1})}{|z-x|^2}|x|\,dz\\
&=\int_E\frac{(G(z_1,y_1)-G(x_1,y_1))|z_1x_1|^2}{p^{-2(m-1)}|z_1-x_1|^2}\left|\frac{p^{m-1}}{x_1}\right||p^{m-1}|\,\frac{dz_1}{|z_1|^2}\\
&=\delta(x_1,y_1)-\frac{1}{V}\\
&=\delta(x,y)-\frac{1}{V}
\end{align*}
Therefore $G_1(x,y)=G(x,y)$ by the uniqueness theorem \eqref{uniqueness}.

\end{proof}

\section{Determining \texorpdfstring{$B_{p,m}(x,y)$}{B\textunderscore p,m(x,y)}} \label{sec:DeterminingB}

\subsection{\texorpdfstring{$D$, $G$, $B$, and $C$}{D, G, B, and C}}

Recall that
\begin{equation}
    D\phi(x,y)=\int_{\bigcup_{s=0}^{m-1}p^s\mathbb{Z}_p^\times}\dfrac{\phi(z,y)-\phi(x,y)}{|z-x|^2}|x|\,dz\label{E:dphixy}
\end{equation}
and define
\begin{equation}
    d(x,y)=\dfrac{|x-y|}{\max\{|x|,|y|\}}
\end{equation}
We prove that $G_{p,m}(x,y)=B_{p,m}(x,y)+C_{p,m}(x,y)$, where $B_{p,m}(x,y)$ is an infinite series in $d(x,y)$, depending only on $|x|$, $|y|$, and $|x-y|$ (i.e., only on $v_p(x)$, $v_p(y)$, and $v_p(x-y)$), with a log singularity on the diagonal; $C_{p,m}(x,y)$ depends only on $|x|$ and $|y|$; and both $B_{p,m}(x,y)$ and $C_{p,m}(x,y)$ are symmetric. In the $m=1$ case, $G_{p,m}=B_{p,m}$; $C_{p,m}$ is a `correction' to $B_{p,m}$ to account for the dependence of $DB_{p,m}(x,y)$ on $|x|$ and $|y|$. Note that both $B_{p,m}$ and $C_{p,m}$ are locally constant off the diagonal.

The following lemma is key in proving $G_{p,m}=B_{p,m}+C_{p,m}$:

\begin{lemma}\label{GOffDiagLemma}
Suppose $G(x,y)$ is a locally constant function on $E\times E$ off the diagonal, with a log singularity at the diagonal, and suppose $DG(x,y)=-\frac{1}{V}$ off the diagonal; then $G$ defines a distribution on the space of continuous functions on the Tate curve, and $G(x,y)$ satisfies $DG(x,y)=\delta(x,y)-\frac{1}{V}$. I.e., $G(x,y)$ is a Green's function for $D$. 

\end{lemma}
\begin{proof}
It suffices to prove the result for the space of locally constant functions on the Tate curve, as they are dense in the space of continuous functions.

First, it was proven in Lemma \ref{DPresLCLemma} that $D$ preserves locally constant functions. Next, for a function $G$ satisfying the above assumptions, $G$ defines a distribution on the space of locally constant functions, by specifying the integration kernel to be $G$. The integrals converge as a result of the log singularity of $G$. Therefore, as $D$ is self-adjoint, $DG$ is again a distribution on the space of locally constant functions. Fix any $y$. By the above assumptions on $G$, $DG(x,y)+\frac{1}{V}$ is a distribution supported at the single point $y$. On the other hand, the space of function germs at $y$ is 1-dimensional, since we are working with locally constant functions. Therefore, there is a constant $c_y$ such that $DG(x,y)=c_y\delta(x,y)-\frac{1}{V}$. Integrating both sides with respect to $x$ over $E$, as $D$ is self-adjoint, and as the constant function $1$ is in the kernel of $D$, the integral on the left hand side gives zero. From this we deduce that $c_y=1$.
\end{proof}

By Lemma \ref{GOffDiagLemma}, to prove $G_{p,m}=B_{p,m}+C_{p,m}$, it suffices to prove that $DB_{p,m}+DC_{p,m}=-\frac{1}{V}$ when $x\neq y$. The following explains our idea to find such $B_{p,m}$ and $C_{p,m}$: As $C_{p,m}(x,y)$ can only depend on $|x|$ and $|y|$, the same must hold for $DC_{p,m}(x,y)$. So if $G_{p,m}=B_{p,m}+C_{p,m}$, it follows that the same must hold for $DB_{p,m}(x,y)$ as well. So, when $x\neq y$, $DB_{p,m}(x,y)$ cannot depend on $|x-y|$. Thus, if we first find $B_{p,m}$ such that (for $x\neq y$) $DB_{p,m}$ depends only on $|x|$ and $|y|$, and then find $C_{p,m}$ such that $DB_{p,m}+DC_{p,m}=-\frac{1}{V}$, then $B_{p,m}+C_{p,m}$ must (up to a constant) be the unique symmetric Green's function. In section \ref{sec:DeterminingB}, we find $B_{p,m}$ such that $DB_{p,m}$ depends only on $|x|$ and $|y|$; in section \ref{subsec:CExistence}, we prove the existence of $C_{p,m}$, which finishes the proof that, up to a constant, $G_{p,m}=B_{p,m}+C_{p,m}$ (section \ref{subsec:ComputingC} concerns other properties of $C_{p,m}$, including explicit recursive formulas for $C_{p,m}(x,y)$).

Note that $d(x,y)=1$ when $|x|\neq|y|$. We shall prove that, up to a constant,
\begin{equation}
    \label{E:Bpmxy}
    B_{p,m}(x,y)=A(y)+\lambda_0\log(d(x,y))+\sum_{n=1}^{\infty}\lambda_n(y)\,d(x,y)^n
\end{equation}
where $\lambda_0\in\mathbb{R}$ is a constant and $A(y),\lambda_n(y)\in\mathbb{R}$ depend only on $|y|$.\footnote{It is clear from \eqref{E:dphixy} that adding a function of $y$ to $\phi(x,y)$ does not affect $D\phi(x,y)$; $A(y)$ serves to symmetrize $B_{p,m}(x,y)$.} We compute $D\log(d(x,y))$ and $Dd(x,y)^n$, and then use the requirement that $DB_{p,m}(x,y)$ cannot depend on $|x-y|$ (disregarding $x=y$) to determine the $\lambda_n(y)$ for $n>0$. Put in another way, it turns out that there exist coefficients $\lambda_n(y)$ such that $DB_{p,m}(x,y)$ does not depend on $|x-y|$.\footnote{Separately computing $D$ for each `term' yields $DB_{p,m}(x,y)$ in the form of an infinite sum, and we prove in section \ref{subsec:DBAbsConv} that the sum is absolutely convergent. Thus, computing each term separately gives the same result as computing $DB_{p,m}(x,y)$ without first separating $B_{p,m}(x,y)$ into different terms, and $DB_{p,m}(x,y)$ is well-defined.}

\subsection{The \texorpdfstring{$m=1$}{m=1} Case}

Aside from this subsection, we will usually take $m>1$. The $m=1$ case is similar to the $m>1$ case, and in many ways simpler (e.g., we can take $C_{p,1}(x,y)\equiv0$ and $A(y)=0$, we do not need different $\lambda_n$ for different $|y|$, $U(y)$ as defined in \eqref{E:Uy} is always equal to 2, and $d(x,y)$ becomes $|x-y|$). The computation of $DB_{p,m}(x,y)$ for the $m=1$ case largely follows the $v_p(x)=v_p(y)=v_p(z)$ portion of the $m>1$ case, and \eqref{E:DlogdxyA}, \eqref{E:DdxyA}, \eqref{E:DdxynA}, \eqref{E:lambdan}, \eqref{E:DBxyA}, and \eqref{E:lambda0} all hold for $m=1$.

\subsection{Terms of \texorpdfstring{$DB_{p,m}(x,y)$}{DB\textunderscore p,m(x,y)}}

\subsubsection{Computing \texorpdfstring{$D\log(d(x,y))$}{D log(d(x,y))}}

We first compute $D\log(d(x,y))$. For $x\neq y$,
\begin{equation}
    D\log(d(x,y))=\int_{\bigcup_{s=0}^{m-1}p^s\mathbb{Z}_p^\times}\dfrac{\log(d(z,y))-\log(d(x,y))}{|z-x|^2}|x|\,dz\label{E:Dlogdxy}
\end{equation}

Suppose $v_p(x)=q$ and $v_p(y)=r$ (i.e., $x\in p^q\mathbb{Z}_p^\times$ and $y\in p^r\mathbb{Z}_p^\times$) for $q\neq r$. Then, $\log(d(x,y))=0$, and $\log(d(z,y))=0$ unless $v_p(z)=r$, and the integrand in \eqref{E:Dlogdxy} vanishes unless $v_p(z)=r$. So, \eqref{E:Dlogdxy} becomes
\begin{equation}
    D\log(d(x,y))=\int_{p^r\mathbb{Z}_p^\times}\dfrac{r-v_p(z-y)}{\max\{p^{-2q},p^{-2r}\}}p^{-q}\,dz
\end{equation}
Define
\begin{equation}
    \partial y_\ell=\left\{z\in\bigcup_{s=0}^{m-1}p^s\mathbb{Z}_p^\times\,\middle|\,v_p(z)=v_p(y),v_p(z-y)=\ell\right\}
\end{equation}
Then, we can write $p^r\mathbb{Z}_p^\times\setminus\{y\}=\bigcup_{i=r}^{\infty}\partial y_i$. Note that $\mu^+(\partial y_r)=(p-2)p^{-r-1}$ and $\mu^+(\partial y_i)=(p-1)p^{-i-1}$ for $i>r$.
So, for $v_p(x)=q\neq r=v_p(y)$,
\begin{align}
    D\log(d(x,y))&=\int_{\bigcup_{i=r}^{\infty}\partial y_i}\dfrac{r-v_p(z-y)}{\max\{p^{-2q},p^{-2r}\}}p^{-q}\,dz\\
    &=\sum_{i=r}^{\infty}\int_{\partial y_i}\dfrac{r-v_p(z-y)}{\max\{p^{-2q},p^{-2r}\}}p^{-q}\,dz\nonumber\\
    &=\sum_{i=r+1}^{\infty}\dfrac{r-i}{\max\{p^{-2q},p^{-2r}\}}p^{-q}(p-1)p^{-i-1}\label{E:Dlogdqr}
\end{align}
Evaluating \eqref{E:Dlogdqr}, we obtain that, for $v_p(x)=q\neq r=v_p(y)$,
\begin{equation}
    D\log(d(x,y))=-\dfrac{p^{-|q-r|}}{p-1}\label{E:DlogdqrA}
\end{equation}

Suppose $v_p(x)=v_p(y)=r$ (i.e., $x,y\in p^r\mathbb{Z}_p^\times$) and $v_p(x-y)=\ell$. As $\mu^+(p^s\mathbb{Z}_p^\times)=(p-1)p^{-s-1}$, for $s\neq r$ we have
\begin{align}
    \int_{p^s\mathbb{Z}_p^\times}\dfrac{\log(d(z,y))-\log(d(x,y))}{|z-x|^2}|x|\,dz&=\int_{p^s\mathbb{Z}_p^\times}\dfrac{\ell-r}{\max\{p^{-2r},p^{-2s}\}}p^{-r}\,dz\nonumber\\
    &=\dfrac{\ell-r}{\max\{p^{-2r},p^{-2s}\}}(p-1)p^{-r-s-1}\label{E:Dlogdrrs}
\end{align}
We also have
\begin{equation}
    \int_{p^r\mathbb{Z}_p^\times}\dfrac{\log(d(z,y))-\log(d(x,y))}{|z-x|^2}|x|\,dz=\sum_{i=r}^{\infty}\int_{\partial x_i}\dfrac{\ell-v_p(z-y)}{p^{-2i}}p^{-r}\,dz\label{E:Dlogdrr1}
\end{equation}
If $i>\ell$ (i.e., $|z-x|<|x-y|$), then $|x-y|=|z-y|$ and the integrand in the right hand side of \eqref{E:Dlogdrr1} vanishes on $\partial x_i$. If $i<\ell$, then $v_p(z-y)=i$ for $z\in\partial x_i$. Thus, the right hand side of \eqref{E:Dlogdrr1} becomes
\begin{equation}
    \sum_{i=r}^{\ell-1}\int_{\partial x_i}\dfrac{\ell-i}{p^{-2i}}p^{-r}\,dz+\int_{\partial x_\ell}\dfrac{\ell-v_p(z-y)}{p^{-2\ell}}p^{-r}\,dz\label{E:Dlogdrr2}
\end{equation}
As $v_p(x-y)=\ell$, we can write $\partial x_\ell=(\partial y_\ell\cap\partial x_\ell)\cup\left(\bigcup_{i=\ell+1}^{\infty}\partial y_i\right)$. As the integrand in the second term of \eqref{E:Dlogdrr2} vanishes on $\partial y_\ell\cap\partial x_\ell$, \eqref{E:Dlogdrr2} becomes
\begin{align}
    &\sum_{i=r}^{\ell-1}\int_{\partial x_i}\dfrac{\ell-i}{p^{-2i}}p^{-r}\,dz+\sum_{i=\ell+1}^{\infty}\int_{\partial y_i}\dfrac{\ell-i}{p^{-2\ell}}p^{-r}\,dz\nonumber\\
    ={}&\dfrac{p-2}{p}(\ell-r)+\sum_{i=r+1}^{\ell-1}(\ell-i)(p-1)p^{i-r-1}+\sum_{i=\ell+1}^{\infty}(\ell-i)(p-1)p^{2\ell-i-r-1}\nonumber\\
    ={}&-\dfrac{1}{p-1}-\dfrac{2}{p}(\ell-r)\label{E:Dlogdrrr}
\end{align}
Combining \eqref{E:Dlogdrrs} and \eqref{E:Dlogdrrr}, we obtain (for $v_p(x)=v_p(y)=r$ and $v_p(x-y)=\ell$)
\begin{align}
    &\int_{\bigcup_{s=0}^{m-1}p^s\mathbb{Z}_p^\times}\dfrac{\log(d(z,y))-\log(d(x,y))}{|z-x|^2}|x|\,dz\nonumber\\
    ={}&\sum_{s=0}^{r-1}(\ell-r)(p-1)p^{s-r-1}-\dfrac{1}{p-1}-\dfrac{2}{p}(\ell-r)+\sum_{s=r+1}^{m-1}(\ell-r)(p-1)p^{r-s-1}\nonumber\\
    ={}&-\dfrac{1}{p-1}-\left(p^{r-m}+p^{-r-1}\right)(\ell-r)\label{E:DlogdrrA}
\end{align}

Define
\begin{equation}
    U(y)=p^{-v_p(y)}+p^{-m+v_p(y)+1}=|y|+p^{1-m}|y|^{-1}\label{E:Uy}
\end{equation}
(we will sometimes write $U(y)$ as $U(r)$ for $r=v_p(y)$).

Taken together, \eqref{E:DlogdqrA} and \eqref{E:DlogdrrA} give
\begin{equation}
    D\log(d(x,y))=\begin{cases}
        -\dfrac{p^{-|v_p(x)-v_p(y)|}}{p-1} & v_p(x)\neq v_p(y)\\[15pt]
        -\dfrac{1}{p-1}-\dfrac{U(y)}{p}(v_p(x-y)-v_p(y)) & v_p(x)=v_p(y)
    \end{cases}\label{E:DlogdxyA}
\end{equation}

\subsubsection{\texorpdfstring{Computing $Dd(x,y)^n$}{Dd(x,y)\^{}n} for \texorpdfstring{$n\geq1$}{n\textgreater=1}}

The computation of $Dd(x,y)^n$ for $n\geq1$ is similar to $D\log(d(x,y))$; the full computations are shown in appendix A.

We obtain
\begin{equation}
    Dd(x,y)=\begin{cases}
        -\dfrac{p^{-|v_p(x)-v_p(y)|}}{p+1} & v_p(x)\neq v_p(y)\\[15pt]
        \begin{aligned}
            -\dfrac{1}{p+1}-\dfrac{U(y)}{p}\left(1-p^{-(v_p(x-y)-v_p(y))}\right)\\
            +\dfrac{p-1}{p}(v_p(x-y)-v_p(y))
        \end{aligned} & v_p(x)=v_p(y)
    \end{cases}\label{E:DdxyA}
\end{equation}

And, for $n>1$,
\begin{equation}
    Dd(x,y)^n=\begin{cases}
        -\dfrac{p^n-1}{p^{n+1}-1}p^{-|v_p(x)-v_p(y)|} & v_p(x)\neq v_p(y)\\[15pt]
        \begin{aligned}
            \dfrac{p^{n-1}-p^{-1}}{p^{n-1}-1}-\dfrac{U(y)}{p}\left(1-p^{-n(v_p(x-y)-v_p(y))}\right)\\
            -\dfrac{(p+1)(p^n-1)^2}{p(p^{n-1}-1)(p^{n+1}-1)}p^{(1-n)(v_p(x-y)-v_p(y))}
        \end{aligned} & v_p(x)=v_p(y)
    \end{cases}\label{E:DdxynA}
\end{equation}

\subsection{Computing the \texorpdfstring{$\lambda_n(y)$}{lambda\textunderscore n(y)}} \label{subsec:lambdan}

Recall that (disregarding $x=y$) $DB_{p,m}(x,y)$ can depend only on $v_p(x)$ and $v_p(y)$; however, each `term' of $DB_{p,m}(x,y)$ has an explicit dependence on $v_p(x-y)$ when $v_p(x)=v_p(y)$. So, the $\lambda_n(y)$ must be such that this dependence is eliminated. Examining \eqref{E:DlogdxyA}, \eqref{E:DdxyA}, and \eqref{E:DdxynA}, we see that we must have $\lambda_1(y)=U(y)\dfrac{1}{p-1}\lambda_0$, and for $n>1$,
\begin{equation}
    \lambda_n(y)=U(y)\dfrac{\left(p^{n-1}-1\right)\left(p^{n+1}-1\right)}{(p+1)(p^n-1)^2}\lambda_{n-1}(y)
\end{equation}
Equivalently, for $n>0$,
\begin{equation}
    \lambda_n(y)=U(y)^n\dfrac{p^{n+1}-1}{(p-1)(p+1)^n(p^n-1)}\lambda_0\label{E:lambdan}
\end{equation}
We set
\begin{equation}
    A(y)=-\sum_{n=1}^{\infty}\lambda_n(y)\label{E:Ay}
\end{equation}
so that $B_{p,m}(x,y)$ is symmetric for $v_p(x)\neq v_p(y)$ (we then have $B_{p,m}(x,y)\equiv0$ when $v_p(x)\neq v_p(y)$).

\subsection{Computing \texorpdfstring{$DB_{p,m}(x,y)$}{DB\textunderscore p,m(x,y)}} \label{subsec:CompDB}

Recall
\begin{equation*}\tag{\ref{E:Uy}}
    U(y)=p^{-v_p(y)}+p^{-m+v_p(y)+1}=|y|+p^{1-m}|y|^{-1}
\end{equation*}
By abuse of notation, we sometimes denote $U(y)$ by $U(v_p(y))$.

Putting together \eqref{E:DlogdxyA}, \eqref{E:DdxyA}, \eqref{E:DdxynA}, and \eqref{E:lambdan}, when $v_p(x)\neq v_p(y)$ we have
\begin{align}
    DB_{p,m}(x,y)&=-\lambda_0p^{-|v_p(x)-v_p(y)|}\sum_{n=0}^{\infty}\dfrac{U(y)^n}{(p-1)(p+1)^n}\\
    &=-\lambda_0\dfrac{p^{-|v_p(x)-v_p(y)|}}{p-1}\dfrac{p+1}{p+1-U(y)}\label{E:DBqrA}
\end{align}

When $v_p(x)=v_p(y)=r$ (but $x\neq y$) and $v_p(x-y)=\ell$, we substitute in \eqref{E:lambdan}, eliminate terms that depend on $v_p(x-y)$, regroup terms by powers of $U(y)$, and simplify to obtain
\begin{align}
    &DB_{p,m}(x,y)\nonumber\\
    =&-\lambda_0\left[\dfrac{1}{p-1}+\dfrac{U(y)}{p}(\ell-r)\right]-\lambda_1(y)\left[\dfrac{1}{p+1}+\dfrac{U(y)}{p}\left(1-p^{-(\ell-r)}\right)-\dfrac{p-1}{p}(\ell-r)\right]\nonumber\\
    &-\sum_{n=2}^{\infty}\lambda_n(y)\left[-\dfrac{p^{n-1}-p^{-1}}{p^{n-1}-1}+\dfrac{U(y)}{p}\left(1-p^{-n(\ell-r)}\right)+\dfrac{(p+1)(p^n-1)^2}{p(p^{n-1}-1)(p^{n+1}-1)}p^{(1-n)(\ell-r))}\right]\nonumber\\
    =&-\lambda_0\left[\dfrac{1}{p-1}\right]-\lambda_0\left[\dfrac{U(y)}{p-1}\left(\dfrac{1}{p+1}+\dfrac{U(y)}{p}\right)\right]\nonumber\\
    &-\lambda_0\sum_{n=2}^{\infty}\left[U(y)^n\dfrac{p^{n+1}-1}{(p-1)(p+1)^n(p^n-1)}\left(-\dfrac{p^{n-1}-p^{-1}}{p^{n-1}-1}+\dfrac{U(y)}{p}\right)\right]\nonumber\\
    =&-\lambda_0\sum_{n=0}^{\infty}\dfrac{U(y)^n}{(p-1)(p+1)^n}\nonumber\\
    =&-\lambda_0\dfrac{1}{p-1}\dfrac{p+1}{p+1-U(y)}\label{E:DBrrA}
\end{align}

From \eqref{E:DBqrA} and \eqref{E:DBrrA}, whenever $x\neq y$,
\begin{equation}
    DB_{p,m}(x,y)=-\lambda_0\dfrac{p^{-|v_p(x)-v_p(y)|}}{p-1}\dfrac{p+1}{p+1-U(y)}\label{E:DBxyA}
\end{equation}

In section \ref{subsec:CExistence}, we prove
\begin{equation*}\tag{\ref{E:lambda0}}
    \lambda_0=\frac{p(p-1)}{p+1}
\end{equation*}

\subsection{Absolute convergence of \texorpdfstring{$DB_{p,m}(x,y)$}{DB\textunderscore p,m(x,y)}} \label{subsec:DBAbsConv}

Our arguments in sections \ref{subsec:lambdan} and \ref{subsec:CompDB} rely on $DB_{p,m}(x,y)$ being absolutely convergent for $x\neq y$; we now prove that this is indeed the case.

When we apply $D$ to $d(x,y)^n$ for any $n\geq1$, the resulting integrand is always bounded. Because the domain of integration has finite measure, all of the resulting integrals are absolutely convergent. For $D\log(d(x,y))$, our integrand is not always bounded, but the resulting integrals are still absolutely convergent, as
\begin{equation*}
    \mu^+\left(\left\{z\in\bigcup_{s=0}^{m-1}p^s\mathbb{Z}_p^\times\,\middle|\,\left|\dfrac{\log(d(z,y))-\log(d(x,y))}{|z-x|^2}\right|_\mathbb{R}>M\right\}\right)
\end{equation*}
decreases exponentially as $M$ increases. So, each `term' of $DB_{p,m}(x,y)$ is individually absolutely convergent.

Next, $|U(y)|_\mathbb{R}\leq 1+p^{1-m}\leq2$, so (for $n>0$)
\begin{equation*}
    |\lambda_n(y)|_\mathbb{R}\leq2^n\dfrac{p^{n+1}-1}{(p-1)(p+1)^n(p^n-1)}\leq\dfrac{2^n}{(p-1)(p+1)^{n-1}}\leq2^n3^{1-n}
\end{equation*}
When $v_p(x)\neq v_p(y)$, we have
\begin{equation*}
    |Dd(x,y)^n|_\mathbb{R}\xrightarrow[n\to\infty]{}p^{-|v_p(x)-v_p(y)|-1}\leq p^{-2}
\end{equation*}
so $DB_{p,m}(x,y)$ is absolutely convergent when $v_p(x)\neq v_p(y)$. When $v_p(x)=v_p(y)\neq v_p(x-y)$,
\begin{equation*}
    |Dd(x,y)^n|_\mathbb{R}\xrightarrow[n\to\infty]{}\left|\dfrac{p^{n-1}-p^{-1}}{p^{n-1}-1}-\dfrac{U(y)}{p}\right|_\mathbb{R}<\dfrac{p+1}{p}
\end{equation*}
and so $DB_{p,m}(x,y)$ is absolutely convergent here as well. When $v_p(x)=v_p(y)=v_p(x-y)$,
\begin{equation*}
    |Dd(x,y)^n|_\mathbb{R}\xrightarrow[n\to\infty]{}\left|\dfrac{p^{n-1}-p^{-1}}{p^{n-1}-1}-\dfrac{(p+1)(p^n-1)^2}{p(p^{n-1}-1)(p^{n+1}-1)}\right|_\mathbb{R}<\dfrac{(p+1)^2}{p^2}
\end{equation*}
and therefore $DB_{p,m}(x,y)$ is absolutely convergent whenever $x\neq y$.

\section{Properties of \texorpdfstring{$C_{p,m}(x,y)$}{C\textunderscore p,m(x,y)}}

We now investigate the properties of $C_{p,m}(x,y)$. First, we prove the existence of $C_{p,m}(x,y)$.

\subsection{Existence of \texorpdfstring{$C_{p,m}(x,y)$}{C\textunderscore p,m(x,y)}} \label{subsec:CExistence}

Recall that the Green's function satisfies
\begin{align}
DG(x,y)=\delta(x,y)-\frac{1}{V}
\label{S:eq1}
\end{align}
where $V=\mu^{\times}(\mathbb{Q}_{p}^{\times}/p^{m\mathbb{Z}})=\mu^{\times}(\bigcup_{k=0}^{m-1}p^{k}\mathbb{Z}_{p}^{\times})=m(1-p^{-1})$. \eqref{S:eq1} can be rewritten as
\begin{align}
DC_{p,m}(x,y)=-DB_{p,m}(x,y)-\frac{1}{m(1-p^{-1})}
\label{S:eq2}
\end{align}
From \eqref{E:DBxyA} we know that
\begin{align*}
DB_{p,m}(x,y)=-\frac{p+1}{p-1}\frac{|x||y|}{\text{max}\{|x|,|y|\}^2}\frac{\lambda_0}{p+1-|y|-p^{1-m}|y|^{-1}}
\end{align*}
when $x\neq y$.

As $C_{p,m}$ depends only on $i:=v_{p}(x)$ and $j:=v_{p}(y)$, we write $(\boldsymbol{C})_{ij}=c_{ij}=C_{p,m}(x,y)=C_{p,m}(v_{p}(x),v_{p}(y))$. For $v_{p}(y)=\ell$ with $0\leq\ell\leq m-1$, we have
\begin{align*}
&DC_{p,m}(x,y)\\
={}&|x|\int_{\mathbb{Q}_{p}^{\times}/p^{m\mathbb{Z}}}\frac{C_{p,m}(z,y)-C_{p,m}(x,y)}{|z-x|^2}\,dz \\
={}&|x|\sum_{k=0}^{m-1}\int_{p^{k}\mathbb{Z}_{p}^{\times}}\frac{C_{p,m}(z,y)-C_{p,m}(x,y)}{|z-x|^2}\,dz \\
={}&|x|\sum_{k=0,k\neq v_{p}(x)}^{m-1}\frac{c_{k\ell}-c_{v_{p}(x)\ell}}{\text{max}\{p^{-k},|x|\}^2}p^{-k}(1-p^{-1}) \\
={}&|x|(1-p^{-1})\left(\sum_{k=0}^{v_{p}(x)-1}p^{k}(c_{k\ell}-c_{v_{p}(x)\ell})+\sum_{k=v_{p}(x)+1}^{m-1}p^{-k}(c_{k\ell}-c_{v_{p}(x)\ell})|x|^{-2}\right) \\
={}&\sum_{k=0}^{v_{p}(x)-1}(p-1)|x|p^{k-1}c_{k\ell}+\frac{|x|+|x|^{-1}p^{1-m}-2}{p|x|}c_{v_{p}(x)\ell}+\sum_{k=v_{p}(x)+1}^{m-1}\frac{p-1}{|x|p^{k+1}}c_{k\ell}
\end{align*}
and we can solve this linear system to obtain $\boldsymbol{C}$.

We denote by $\boldsymbol{A}$ the coefficient matrix of $\boldsymbol{C}$ on the RHS of the above equation, and by $\boldsymbol{B}$ the right hand side of \eqref{S:eq2}, where $0\leq i,j\leq m-1$. Then the equation can be written as $\boldsymbol{A}\boldsymbol{C}=\boldsymbol{B}$, where
\begin{align*}
\boldsymbol{A}_{ij}=
\begin{cases} 
(1-p^{-1})p^{-|i-j|}, & i\neq j \\
p^{-1}(p^{-i}+p^{i+1-m}-2), & i= j 
\end{cases}\\
(DB_{p,m})_{ij}=-\frac{p^{-|i-j|}}{p+1-p^{-j}-p^{j+1-m}}\frac{p+1}{p-1}\lambda_0
\end{align*}

The kernel of $\boldsymbol{A}$ consists of constant functions. I.e., each column of $\boldsymbol{A}$ sums to zero. The solvability condition for this linear system is that the same holds for $\boldsymbol{B}$; that is,
\begin{align*}
\sum_{i=0}^{m-1}\left(\frac{p^{-|i-j|}}{p+1-p^{-j}-p^{j+1-m}}\frac{p+1}{p-1}\lambda_0-\frac{1}{m(1-p^{-1})}\right)=0,\ \forall\ 0\leq j\leq m-1
\end{align*}
this is equivalent to
\begin{equation}\label{E:lambda0}
    \lambda_0=\frac{p(p-1)}{p+1}
\end{equation}
and thus
\begin{align*}
(DB_{p,m})_{ij}=-\frac{p^{1-|i-j|}}{p+1-p^{-j}-p^{j+1-m}}
\end{align*}
On the other hand, with $\lambda_0$ as in \eqref{E:lambda0}, we can always find $C_{p,m}$ such that $DB_{p,m}+DC_{p,m}=-\frac{1}{V}$ away from the diagonal, and therefore $G_{p,m}=B_{p,m}+C_{p,m}$ is a Green's function.

\subsection{Computing \texorpdfstring{$C_{p,m}(x,y)$}{C\textunderscore p,m(x,y)}} \label{subsec:ComputingC}

The computation of $C_{p,m}$ is aided by a sequence of linear algebra lemmas over $\mathbb{R}$.
\begin{lemma}
\label{S:lemma1}
Suppose $\boldsymbol{A},\boldsymbol{B}$ are $n\times n$ centrosymmetric matrices with $\mathrm{rank}(\boldsymbol{A})\geq n-1$ and $\boldsymbol{Z}$ is an $n\times n$ symmetric matrix that satisfies $\boldsymbol{A}\boldsymbol{Z}=\boldsymbol{B}$. Then $\boldsymbol{Z}$ is centrosymmetric.
\end{lemma}

\begin{proof}
As $\boldsymbol{B}$ is centrosymmetric, we have
\begin{align}
\boldsymbol{A}\boldsymbol{Z}=\boldsymbol{B}
=\boldsymbol{J}\boldsymbol{B}\boldsymbol{J}
=\boldsymbol{J}\boldsymbol{A}\boldsymbol{Z}\boldsymbol{J} \label{S:lem1_eq1}
\end{align}
where $\boldsymbol{J}=\begin{pmatrix}  & & & 1 \\  & & 1 & \\ & \iddots & &\\ 1 & & &\end{pmatrix}$. As $\boldsymbol{A}$ is centrosymmetric, we have $\boldsymbol{A}=\boldsymbol{J}\boldsymbol{A}\boldsymbol{J}$,  $\boldsymbol{A}\boldsymbol{J}=\boldsymbol{J}\boldsymbol{A}$;

so we have $\boldsymbol{A}\boldsymbol{Z}
=\boldsymbol{J}\boldsymbol{A}\boldsymbol{Z}\boldsymbol{J}
=\boldsymbol{A}\boldsymbol{J}\boldsymbol{Z}\boldsymbol{J}$, thus
\begin{align*}
\boldsymbol{A}(\boldsymbol{Z}-\boldsymbol{J}\boldsymbol{Z}\boldsymbol{J})=\boldsymbol{0}
\end{align*}

If $\text{rank}(\boldsymbol{A})=n$, then $\boldsymbol{Z}-\boldsymbol{J}\boldsymbol{Z}\boldsymbol{J}=\boldsymbol{0}$, and thus $\boldsymbol{Z}=\boldsymbol{J}\boldsymbol{Z}\boldsymbol{J}$, so $\boldsymbol{Z}$ is centrosymmetric.

If $\text{rank}(\boldsymbol{A})=n-1$, then we have $\text{rank}(\boldsymbol{A})+\text{rank}(\boldsymbol{Z}-\boldsymbol{J}\boldsymbol{Z}\boldsymbol{J})\leq n$, thus 
\begin{align*}
\text{rank}(\boldsymbol{Z}-\boldsymbol{J}\boldsymbol{Z}\boldsymbol{J})\leq 1
\end{align*}

If $\text{rank}(\boldsymbol{Z}-\boldsymbol{J}\boldsymbol{Z}\boldsymbol{J})=0$, then $\boldsymbol{Z}=\boldsymbol{J}\boldsymbol{Z}\boldsymbol{J}$, so $\boldsymbol{Z}$ is centrosymmetric.

If $\text{rank}(\boldsymbol{Z}-\boldsymbol{J}\boldsymbol{Z}\boldsymbol{J})=1$, then $\boldsymbol{Z}-\boldsymbol{J}\boldsymbol{Z}\boldsymbol{J}=\boldsymbol{v}^\top\boldsymbol{u}$, where $\boldsymbol{u}, \boldsymbol{v}$ are nonzero $n$-dimensional vectors. As $\boldsymbol{Z}$ is symmetric, so is $\boldsymbol{Z}-\boldsymbol{J}\boldsymbol{Z}\boldsymbol{J}$, which implies $\boldsymbol{v}=\mu\boldsymbol{u}$, where $\mu \neq0$ is a constant. Therefore, 
\begin{align}
\boldsymbol{Z}-\boldsymbol{J}\boldsymbol{Z}\boldsymbol{J}=\mu\boldsymbol{u}\boldsymbol{u}^\top
\label{S:lem1_eq2}
\end{align}
and thus
\begin{align}
\boldsymbol{J}\boldsymbol{Z}\boldsymbol{J}-\boldsymbol{Z}=\boldsymbol{J}(\boldsymbol{Z}-\boldsymbol{J}\boldsymbol{Z}\boldsymbol{J})\boldsymbol{J}=\mu\boldsymbol{J}\boldsymbol{u}\boldsymbol{u}^\top\boldsymbol{J}
\label{S:lem1_eq3}
\end{align}
Adding \eqref{S:lem1_eq2} to \eqref{S:lem1_eq3}, we obtain $\mu(\boldsymbol{u}\boldsymbol{u}^\top+\boldsymbol{J}\boldsymbol{u}\boldsymbol{u}^\top\boldsymbol{J})=\boldsymbol{0}$. As $\mu \neq0$, we have
\begin{align*}
\boldsymbol{u}\boldsymbol{u}^\top+\boldsymbol{J}\boldsymbol{u}\boldsymbol{u}^\top\boldsymbol{J}=\boldsymbol{0}
\end{align*}
 $(\boldsymbol{u}\boldsymbol{u}^\top+\boldsymbol{J}\boldsymbol{u}\boldsymbol{u}^\top\boldsymbol{J})_{ii}=0$,\ $i=1,2,\dots,n$ implies
\begin{align*}
\begin{cases}
    u_1^2+u_n^2=0\\
    u_2^2+u_{n-1}^2=0\\
    \quad\quad\vdots \\
    u_n^2+u_1^2=0
\end{cases}
\end{align*}
and thus $\boldsymbol{u}=\boldsymbol{0}$, a contradiction.
\end{proof}

\begin{remark}\label{bisymmetric}
$\boldsymbol{A}$ is symmetric$\iff\boldsymbol{A}=\boldsymbol{A}^\top$;
$\boldsymbol{A}$ is persymmetric$\iff\boldsymbol{A}=\boldsymbol{J}\boldsymbol{A}^\top\boldsymbol{J}$;
$\boldsymbol{A}$ is centrosymmetric$\iff\boldsymbol{A}=\boldsymbol{J}\boldsymbol{A}\boldsymbol{J}$;
and $\boldsymbol{A}$ is called bisymmetric if it is both symmetric and persymmetric. Among the three, knowing two allows one to deduce the third. 
\end{remark}

\begin{lemma}
\label{S:lemma2}
Let $\boldsymbol{P}$ be a Kac–Murdock–Szegő (KMS) matrix \cite{MR59482}; i.e., $p_{ij}=r^{|i-j|}$, where $r\in(-1,1)$. Define $\boldsymbol{L}=\mathrm{diag}(\boldsymbol{P}\boldsymbol{1})$ (i.e. $L_i:=\boldsymbol{L}_{ii}$ is the sum of the i-th row of $\boldsymbol{P}$). If $\boldsymbol{Z}$ is a matrix that satisfies $(\boldsymbol{P}-\boldsymbol{L})\boldsymbol{Z}=\alpha\left(\boldsymbol{P}\boldsymbol{L}^{-1}-\frac{\boldsymbol{1}\boldsymbol{1}^\top}{n}\right)$, where $\alpha$ is a constant, then $\boldsymbol{Z}$ satisfies $z_{ij}=
\begin{cases} z_{i1}+z_{nj}-z_{n1}, & i\geq j \\ z_{in}+z_{1j}-z_{1n}, & i\leq j \end{cases} $.
\end{lemma}

\begin{proof}
It is equivalent to show
\begin{align*}
\begin{cases} 
z_{ij}-z_{i1}=z_{nj}-z_{n1}, & i\geq j \\
z_{ij}-z_{in}=z_{1j}-z_{1n}, & i\leq j 
\end{cases}
\end{align*}
denoting $
\begin{cases} 
D_{ij}=z_{ij}-z_{i1}, & i\geq j \\
E_{ij}=z_{ij}-z_{in}, & i\leq j 
\end{cases} $, we need to show
\begin{align*}
\begin{cases} 
D_{ij}=D_{nj}, & i\geq j \\
E_{ij}=E_{1j}, & i\leq j 
\end{cases}
\end{align*}
First, we consider the $i\leq j$ case. Our goal is to prove $E_{1j}=E_{2j}=\dots=E_{jj},$ $\forall j=1,2,\dots,n$. As
\begin{align}
\sum_{k=1}^np_{ik}z_{kj}-L_{i}z_{ij}=\alpha\left(\frac{p_{ij}}{L_j}-\frac{1}{n}\right),
\label{S:lem2_eq1}
\end{align}
when we let $j=n$, we have
\begin{align}
\sum_{k=1}^np_{ik}z_{kn}-L_{i}z_{in}=\alpha\left(\frac{p_{in}}{L_n}-\frac{1}{n}\right)
\label{S:lem2_eq2}
\end{align}
Subtracting \eqref{S:lem2_eq2} from \eqref{S:lem2_eq1} gives
\begin{align*}
\sum_{k=1}^np_{ik}E_{kj}-L_{i}E_{ij}=\alpha\left(\frac{p_{ij}}{L_j}-\frac{p_{in}}{L_n}\right)
\end{align*}
As $p_{ij}=r^{|i-j|}$,
\begin{align}
L_{k}=\sum_{i=0}^{n-k}r^{i}+\sum_{i=1}^{k-1}r^{i}
\label{S:lem2_eq3}
\end{align}
so we have
\begin{align}
\sum_{k=1}^{i-1}r^{i-k}E_{kj}+\sum_{k=i}^{n}r^{k-i}E_{kj}-L_{i}E_{ij}=\alpha\left(\frac{r^{j-i}}{L_j}-\frac{r^{n-i}}{L_n}\right)
\label{S:lem2_eq4}
\end{align}
Replacing $i$ with $i+1$ yields
\begin{align}
\sum_{k=1}^{i}r^{i+1-k}E_{kj}+\sum_{k=i+1}^{n}r^{k-i-1}E_{kj}-L_{i+1}E_{i+1,j}=\alpha\left(\frac{r^{j-i-1}}{L_j}-\frac{r^{n-i-1}}{L_n}\right)
\label{S:lem2_eq5}
\end{align}
Multiplying \eqref{S:lem2_eq4} by $r^{-1}$ and then subtracting \eqref{S:lem2_eq5} yields
\begin{align}
\sum_{k=1}^{i}(r^{i-k-1}-r^{i+1-k})E_{kj}
-r^{-1}L_{i}E_{ij}+L_{i+1}E_{i+1,j}=0
\label{S:lem2_eq7}
\end{align}
Next, by induction:\\
When $i=1$, \eqref{S:lem2_eq7} becomes
\begin{align}
(r^{-1}-r)E_{1j}
-r^{-1}L_{1}E_{1j}+L_{2}E_{2j}=0
\label{S:lem2_eq8}
\end{align}
Substitute \eqref{S:lem2_eq3} into \eqref{S:lem2_eq8}, and we have
\begin{equation*}
    \left(r^{-1}-r-\sum_{k=0}^{n-1}r^{k-1}\right)E_{1j}
+L_{2}E_{2j}=0
\end{equation*}
which is
\begin{equation*}
    \left(-r-\sum_{k=0}^{n-2}r^{k}\right)E_{1j}
+L_{2}E_{2j}=0
\end{equation*}
That is,
\begin{align*}
L_{2}(E_{2j}-E_{1j})=0
\end{align*}
and hence $E_{2j}=E_{1j}$.

For the induction step, we assume $E_{1j}=E_{2j}=\dots=E_{ij}$, where $i<j$, and \eqref{S:lem2_eq7} becomes
\begin{align}
\left(\sum_{k=1}^{i}(r^{i-k-1}-r^{i+1-k})-r^{-1}L_{i}\right)E_{ij}+L_{i+1}E_{i+1,j}=0
\label{S:lem2_eq9}
\end{align}
Substitute \eqref{S:lem2_eq3} into \eqref{S:lem2_eq9}, and we have
\begin{equation*}
    \left(\sum_{k=1}^{i}\left(r^{i-k-1}-r^{i+1-k}\right)-r^{-1}\left(\sum_{k=0}^{n-i}r^{k}+\sum_{k=1}^{i-1}r^{k}\right)\right)E_{ij}+L_{i+1}E_{i+1,j}=0
\end{equation*}
I.e.,
\begin{equation*}
    -\left(\sum_{k=0}^{n-i-1}r^{k}+\sum_{k=1}^{i}r^{k}\right)E_{ij}+L_{i+1}E_{i+1,j}=0
\end{equation*}
That is,
\begin{align*}
L_{i+1}(E_{i+1,j}-E_{ij})=0
\end{align*}
and hence $E_{i+1,j}=E_{ij}$.\\
For the $i\geq j$ case, we multiply \eqref{S:lem2_eq4} by $r$, subtract \eqref{S:lem2_eq5}, and use induction, in parallel to the $i\leq j$ case. We obtain $D_{jj}=D_{j+1,j}=\dots=D_{nj}, \forall j=1,2,\dots,n$.
\end{proof}

\begin{remark}
$\boldsymbol{Z}$ does not need to be assumed to be symmetric in Lemma \ref{S:lemma2}.
\end{remark}

\begin{lemma}
\label{S:lemma3}
With the same hypotheses as Lemma \ref{S:lemma2}, $\boldsymbol{Z}$ satisfies $z_{ij}+z_{jk}+z_{ki}=z_{ji}+z_{kj}+z_{ik}$.
\end{lemma}

\begin{proof}
Without loss of generality, we assume $i\geq j\geq k$; then by Lemma \ref{S:lemma2}, we have
\begin{align*}
\begin{cases}
z_{ij}=z_{i1}+z_{nj}-z_{n1}\\
z_{jk}=z_{j1}+z_{nk}-z_{n1}\\
z_{ki}=z_{kn}+z_{1i}-z_{1n}\\
z_{ji}=z_{jn}+z_{1i}-z_{1n}\\
z_{kj}=z_{kn}+z_{1j}-z_{1n}\\
z_{ik}=z_{i1}+z_{nk}-z_{n1}
\end{cases}
\end{align*}
so it suffices to show
\begin{equation*}
    z_{i1}+z_{nj}-z_{n1}+z_{j1}+z_{nk}-z_{n1}+z_{kn}+z_{1i}-z_{1n}=z_{jn}+z_{1i}-z_{1n}+z_{kn}+z_{1j}-z_{1n}+z_{i1}+z_{nk}-z_{n1}
\end{equation*}
This is
\begin{align*}
z_{nj}+z_{j1}-z_{n1}=z_{jn}+z_{1j}-z_{1n}
\end{align*}
which follows from Lemma \ref{S:lemma2}.
\end{proof}

\begin{lemma}
\label{S:theorem1}
Let $\boldsymbol{C}$ be a solution to \eqref{S:eq2}; once we reindex to start from $(i,j)=(1,1)$, $\boldsymbol{C}$ satisfies
$c_{ij}:=\boldsymbol{C}_{ij}=
\begin{cases} c_{i1}+c_{mj}-c_{m1}, & i\geq j \\ c_{im}+c_{1j}-c_{1m}, & i\leq j \end{cases} $.
\end{lemma}

\begin{proof}

Using the notation of Lemma \ref{S:lemma2}, $p_{ij}:=P_{ij}=p^{-|i-j|}$, $\boldsymbol{L}=\text{diag}(\boldsymbol{P}\boldsymbol{1})$. We denote
\begin{equation*}
    \Lambda_{i}:=\frac{p+1-p^{1-i}-p^{i-m}}{p}=1+p^{-1}-p^{-1}U(i)
\end{equation*}
where $U(i)$ is as in \eqref{E:Uy} after the reindexing.

Then $\Lambda_{i}=(1-p^{-1})L_i$.  We have $\boldsymbol{A}=(1-p^{-1})(\boldsymbol{P}-\boldsymbol{L})$ and $D\boldsymbol{B}_{p,m}=-\boldsymbol{P}\boldsymbol{\Lambda^{-1}}$;  \eqref{S:eq2} becomes
\begin{align}
(1-p^{-1})(\boldsymbol{P}-\boldsymbol{L})\boldsymbol{C}=\frac{1}{1-p^{-1}}\left(\boldsymbol{P}\boldsymbol{L}^{-1}-\frac{\boldsymbol{1}\boldsymbol{1}^\top}{m}\right)
\label{S:thm1_eq1}
\end{align}
Then by Lemma \ref{S:lemma2} with $\alpha=\frac{1}{(1-p^{-1})^2}$, $n=m$, and $r=p^{-1}$, it follows that $\boldsymbol{C}$ satisfies $c_{ij}=
\begin{cases} c_{i1}+c_{mj}-c_{m1}, & i\geq j \\ c_{im}+c_{1j}-c_{1m}, & i\leq j \end{cases}$.
\end{proof}

\begin{corollary}
   Equation \eqref{S:eq2} has a solution given by a symmetric $\boldsymbol{C}$, therefore also a bisymmetric $\boldsymbol{C}$.  
\end{corollary}

\begin{proof}
The existence of solutions to \eqref{S:eq2} has already been proved. Suppose $\boldsymbol{C}$ is a solution to equation \eqref{S:eq2}, then after possibly adding suitable constants to its columns, we may assume that the first row of the solution $\boldsymbol{C}$ agrees with the transpose of its first column. By Lemma \ref{S:theorem1}, looking at diagonal entries of $\boldsymbol{C}$, we conclude that the last row of $\boldsymbol{C}$ agrees with the transpose of the last column of $\boldsymbol{C}$. Applying Lemma \ref{S:theorem1} again, we deduce that $\boldsymbol{C}$ is symmetric.

Next, one checks that $\boldsymbol{A}$ is bisymmetric and $\boldsymbol{B}$ is centrosymmetric. So by Lemma \ref{S:lemma1}, $\boldsymbol{C}$ is centrosymmetric, and thus it is bisymmetric by remark \ref{bisymmetric}.
\end{proof}

\begin{theorem} \label{CFormulasThm}
A symmetric $\boldsymbol{C}$ can be computed directly using a set of recursive formulas \eqref{E:CRecur}.
\end{theorem}

\begin{proof}
We have equation \eqref{S:thm1_eq1}
\begin{align*}
(1-p^{-1})\sum_{k=1}^mp_{ik}c_{kj}-\Lambda_{i}c_{ij}=\frac{1}{1-p^{-1}}\left((1-p^{-1})\frac{p_{ij}}{\Lambda_j}-\frac{1}{m}\right)
\end{align*}
Letting $j=1$, for all $1\leq i\leq m-1$ we have
\begin{align}
(1-p^{-1})\left(\sum_{k=1}^{i-1}p^{-(i-k)}c_{k1}+\sum_{k=i}^{m}p^{-(k-i)}c_{k1}\right)-\Lambda_{i}c_{i1}=\frac{1}{1-p^{-1}}\left(\frac{p^{-(i-1)}}{\Lambda_1}-\frac{1}{m}\right)
\label{S:thm2_eq1}
\end{align}
Replacing $i$ with $i+1$ yields
\begin{align}
(1-p^{-1})\left(\sum_{k=1}^{i}p^{-(i+1-k)}c_{k1}+\sum_{k=i+1}^{m}p^{-(k-i-1)}c_{k1}\right)-\Lambda_{i+1}c_{i+1,1}=\frac{1}{1-p^{-1}}\left(\frac{p^{-i}}{\Lambda_1}-\frac{1}{m}\right)
\label{S:thm2_eq2}
\end{align}
Multiply \eqref{S:thm2_eq1} by $p^{-1}$ and subtract \eqref{S:thm2_eq2}, we obtain
\begin{align}
(1-p^{-1})(p^{-1}-p)\sum_{k=i+1}^{m}p^{i-k}c_{k1}
-p^{-1}\Lambda_{i}c_{i1}+\Lambda_{i+1}c_{i+1,1}=\frac{1}{m}
\label{S:thm2_eq3}
\end{align}
Replacing $i$ with $i-1$, we have
\begin{align}
(1-p^{-1})(p^{-1}-p)\sum_{k=i}^{m}p^{i-1-k}c_{k1}
-p^{-1}\Lambda_{i-1}c_{i-1,1}+\Lambda_{i}c_{i1}=\frac{1}{m}
\label{S:thm2_eq4}
\end{align}
Multiplying \eqref{S:thm2_eq4} by $p$ and subtracting \eqref{S:thm2_eq3} gives 
\begin{equation*}
    ((1-p^{-1})(p^{-1}-p)+(p+p^{-1})\Lambda_{i})c_{i1}-\Lambda_{i-1}c_{i-1,1}-\Lambda_{i+1}c_{i+1,1}=\frac{p-1}{m}
\end{equation*}
and so for $2\leq i\leq m-1$,
\begin{align*}
c_{i-1,1}={}&\frac{1}{\Lambda_{i-1}}\left(((1-p^{-1})(p^{-1}-p)+(p+p^{-1})\Lambda_{i})c_{i1}-\Lambda_{i+1}c_{i+1,1}-\frac{p-1}{m}\right)\\
={}&\dfrac{\Lambda_{i-1}+\Lambda_{i+1}}{\Lambda_{i-1}}c_{i,1}-\dfrac{\Lambda_{i+1}}{\Lambda_{i-1}}c_{i+1,1}-\dfrac{p-1}{m\Lambda_{i-1}}
\end{align*}
Taking $i=m-1$ in \eqref{S:thm2_eq3}, we have
\begin{equation*}
    (1-p^{-1})(p^{-1}-p)p^{-1}c_{m1}
-p^{-1}\Lambda_{m-1}c_{m-1,1}+\Lambda_{m}c_{m1}=\frac{1}{m}
\end{equation*}
and thus
\begin{align*}
c_{m-1,1}&=\frac{1}{\Lambda_{m-1}}\left(((1-p^{-1})(p^{-1}-p)+p\Lambda_{m})c_{m1}
-\frac{p}{m}\right)\\
&=c_{m1}-\frac{p}{m\Lambda_{m-1}}
\end{align*}
To summarize, we have the following recursive formulas:
\begin{align*}
\begin{cases}
c_{m-1,1}=c_{m1}-\dfrac{p}{m\Lambda_{m-1}}\\[5pt]
c_{i-1,1}=\dfrac{\Lambda_{i-1}+\Lambda_{i+1}}{\Lambda_{i-1}}c_{i,1}-\dfrac{\Lambda_{i+1}}{\Lambda_{i-1}}c_{i+1,1}-\dfrac{p-1}{m\Lambda_{i-1}}\\[5pt]
c_{ij}=
\begin{cases} 
c_{i1}+c_{m-j+1,1}-c_{m1}, & i\geq j \\ 
c_{m-i+1,1}+c_{j1}-c_{m1}, & i\leq j 
\end{cases}
\end{cases}
\end{align*}
where $\Lambda_{i}=1+p^{-1}-p^{-i}-p^{i-m-1}$.

When we reindex to start from $(i,j)=(0,0)$, this becomes
\begin{equation}
    c_{ij}=\begin{cases}
        c_{m-1,0}-\dfrac{p}{m\Lambda_{m-2}} & i=m-2,j=0\\[10pt]
        \dfrac{\Lambda_i+\Lambda_{i+2}}{\Lambda_i}c_{i+1,0}-\dfrac{\Lambda_{i+2}}{\Lambda_i}c_{i+2,0}-\dfrac{p-1}{m\Lambda_i} & i<m-2,j=0\\[5pt]
        c_{i0}+c_{m-1-j,0}-c_{m-1,0} & i\geq j\\
        c_{m-1-i,0}+c_{j0}-c_{m-1,0} & i\leq j
    \end{cases}\label{E:CRecur}
\end{equation}
where $c_{m-1,0}$ is arbitrary and $\Lambda_i=1+p^{-1}-p^{-1}U(x)=1+p^{-1}-p^{-i-1}-p^{-m+i}$.
\end{proof}
Theorem \ref{CFormulasThm} completes the proof of the Main Theorem \eqref{MainThm}.

For example, if we take $c_{m-1,0}=0$, then $c_{m-2,0}=-\frac{p}{m\Lambda_{m-2}}=-\frac{p^m}{m(p^{m-1}+p^{m-2}-p^{m-3}-1)}$. Explicit numerical values for $C_{p,m}(x,y)$ obtained from these recurrences with $c_{m-1,0}=0$ are provided in appendix B.

\section{Relation with the N\'eron local height function on the Tate curve}
Write $C(v_{p}(x),v_{p}(y))=C_{p,m}(x,y)$.
\begin{corollary}\label{HeightCor}
When $m$ is odd, after choosing the new ``central normalization'' $C(\frac{m-1}{2},\frac{m-1}{2})=0$, we have 
\begin{equation}\label{height}
\lim_{p\to\infty}\left(-\frac{1}{p}G\left(xp^{\frac{m-1}{2}},p^{\frac{m-1}{2}}\right)\right)=h(x)-\frac{m}{12},
\end{equation} where $h(x)$ is the N\'eron local height of the point $x$ \cite{Silverman}. The result generalizes directly to a finite extension of $\mathbb{Q}_p$ due to section \ref{extension}.
\end{corollary}

\begin{proof}
    From \eqref{E:Bpmxy}, \eqref{E:lambda0} and \eqref{E:lambdan}, we have
    \begin{align*}
        \lim_{p\to\infty}\left(-\frac{1}{p}B_{p,m}(x,1)\right)=v_p(x-1)
    \end{align*}
    
    From \eqref{E:CRecur}, by taking the $p\to\infty$ limit in the $0$-th column, we have
    \begin{align*}
    \lim_{p\to\infty}\left(-\frac{1}{p}C(m-1-k,0)\right)=2\lim_{p\to\infty}\left(-\frac{1}{p}C(m-k,0)\right)-\lim_{p\to\infty}\left(-\frac{1}{p}C(m+1-k,0)\right)+\frac{1}{m}
    \end{align*}
    for $1<k\leq m-1$, and also $C(m-1,0)=0$ and $\lim_{p\to\infty}\left(-\frac{1}{p}C(m-2,0)\right)=\frac{1}{m}$.
    Therefore
    \begin{align*}
    \lim_{p\to\infty}\left(-\frac{1}{p}C(m-1-k,0)\right)=\frac{(k+1)k}{2m}
    \end{align*}
    for $0\leq k\leq m-1$.
    
    By Lemma \ref{S:theorem1}, the $\frac{m-1}{2}$-th column of $C$ satisfies:
    \begin{align*}
        &C(k+\tfrac{m-1}{2}\smallpmod{m},\tfrac{m-1}{2})\\
        ={}&\begin{cases}
            C(\tfrac{m-1}{2}+k,0)+C(\tfrac{m-1}{2},0)-C(m-1,0)&0\leq k\leq\frac{m-1}{2}\\
            C(\tfrac{m-1}{2}+m-k,0)+C(\tfrac{m-1}{2},0)-C(m-1,0)&\frac{m+1}{2}\leq k\leq m-1\\
        \end{cases}\\
        ={}&C(\tfrac{m-1}{2}+\tfrac{m}{2}-|k-\tfrac{m}{2}|,0)+C(\tfrac{m-1}{2},0)-C(m-1,0)\\
        ={}&C(m-1-(|k-\tfrac{m}{2}|-\tfrac{1}{2}),0)+C(\tfrac{m-1}{2},0)-C(m-1,0)
    \end{align*}
    After we choose the new normalization $C(\tfrac{m-1}{2},\tfrac{m-1}{2})=0$, we have $C(m-1,0)=2C(\tfrac{m-1}{2},0)$, and thus
    \begin{align*}
        C(k+\tfrac{m-1}{2}\smallpmod{m},\tfrac{m-1}{2})=C(m-1-(|k-\tfrac{m}{2}|-\tfrac{1}{2}),0)-C(\tfrac{m-1}{2},0)
    \end{align*}
    therefore
    \begin{align*}
        &\lim_{p\to\infty}\left(-\frac{1}{p}C(k+\tfrac{m-1}{2}\smallpmod{m},\tfrac{m-1}{2})\right)\\
        ={}&\lim_{p\to\infty}\left(-\frac{1}{p}C(m-1-(|k-\tfrac{m}{2}|-\tfrac{1}{2}),0)\right)-\lim_{p\to\infty}\left(-\frac{1}{p}C(\tfrac{m-1}{2},0)\right)\\
        ={}&\frac{(|k-\frac{m}{2}|-\frac{1}{2})(|k-\frac{m}{2}|+\frac{1}{2})}{2m}-\frac{\frac{m-1}{2}\frac{m+1}{2}}{2m}\\
        ={}&\frac{k^2}{2m}-\frac{k}{2}\\
    \end{align*}
    hence
    \begin{align*}
        &\lim_{p\to\infty} \left(-\frac{1}{p}G\left(xp^{\frac{m-1}{2}},p^{\frac{m-1}{2}}\right)\right)\\
        ={}&\lim_{p\to\infty} \left(-\frac{1}{p}\left(B_{p,m}(x,1)+C(k+\tfrac{m-1}{2}\smallpmod{m},\tfrac{m-1}{2})\right)\right)\\
        ={}&v_p(x-1)+\frac{k^2}{2m}-\frac{k}{2}\\
        ={}&h(x)-\frac{m}{12}
    \end{align*}
\end{proof}

\begin{remark}
This is a non-Archimedean counterpart of the fact that on an elliptic curve over $\mathbb{C}$, the local height function is given by the Green's function. A simultaneous shift in both arguments of the Green's function by a constant as in equation \eqref{height} does not have any effect in the Archimedean case, whereas in the non-Archimedean case, this is needed since the definition of the Laplacian relies on choosing a fundamental domain. As a consequence, the Green's function does not respect this shift symmetry.
\end{remark}

\begin{remark}
If we interpolate the domain of $C$ to allow rational arguments, then
\begin{equation}\label{GH}
\lim_{p\to\infty} \left(-\frac{1}{p}\left(B_{p,m}(x,1)+m^2C(m-\tfrac{k}{m},0)\right)\right)
=v_p(x-1)+\frac{1}{2}k\left(\frac{k}{m}-1\right)
=h(x)-\frac{m}{12}
\end{equation}
where $k=v_p(x)$.

Note that for each fixed $y$, the Green's function $G(x,y)$ uniquely determines $B_{p,m}(x,y)$ and $C_{p,m}(x,y)$ up to an overall shift by a constant. Therefore, the interpolated Green's function determines the local height function for all $m$.
\end{remark}

\begin{remark}
There is a reason behind the elementary transformations of arguments involved in \eqref{GH}: the non-Archimedean N\'eron local height function on the Tate curve, as defined via intersection theory on arithmetic surface given by the Tate curve integral model, involves solving a Green's function on the cycle graph of size $m$, with sources at $0$ and $k$, and then evaluating the Green's function at $k$. On the other hand, our function $C(m-1-k,0)$ is the value of the Green's function for a path graph of size $m$, with sources at $0$ and $k$, where $0$ is the initial vertex of the path. This is because our $G(x,1)$ does not respect the flip symmetry $x\to p^{m-1}x^{-1}$, but instead, $G(x,y)$ respects the double flip symmetry \eqref{double flip}.
\end{remark}

\begin{remark}
In terms of physics, the Green's function on the Tate curve is the two-point function for the proposed $p$-adic string worldsheet action \eqref{action} on the Tate curve, by a standard path integral argument as in chapter 9 of \cite{PS}. This is a non-Archimedean conformal field theory (CFT) of free scalar Bosons in genus one \cite{HMST}. So as $p\to\infty$, the leading term of the two-point function of this CFT gives rise to the local height function. A natural speculation is that this CFT has a bulk dual description on the quotient of the Bruhat-Tits tree $T_p/\Gamma$, where $\Gamma$ is again the discrete subgroup generated by $\begin{bmatrix} q & 0\\0 & 1\end{bmatrix}$. The bulk theory probably contains nontrivial interactions whose contributions to the boundary-to-boundary propagator vanish in the limit $p\to\infty$.

Indeed, in \cite{HJ2024}, a different CFT on the Tate curve is derived from the $p$-adic AdS/CFT correspondence, from the free scalar field action in the bulk $T_p/\Gamma$. Equations (70)-(72) in \cite{HJ2024} show that the Green's function for this different CFT does reproduce the local height function. This is an indication that  different CFTs on the Tate curve may have identical large $p$ limit for the two point function, given by the local height function. Furthermore, this also suggests that one may relate generic features of non-Archimedean CFT two point function defined on various geometry, to the non-Archimedean local height function in a more conceptual way: for the former, develop a non-Archimedean version of portion of the conformal bootstrap at finite temperature, building on \cite{Bootstrap}\cite{HJ2024}, and for the latter, use defining properties of the local height function. 
\end{remark}

We shall defer the detailed investigation of the physics interpretation to a future paper.

\section{Finite extensions of \texorpdfstring{$\mathbb{Q}_p$}{Q\textunderscore p}}\label{extension}

Given $K$ a finite extension of $\mathbb{Q}_p$ with ramification index $e$ and residue field degree $p^f$, we can easily translate all of our computations for $B_{p,m}$ and $C_{p,m}$ to analogous $B_{K,m}$ and $C_{K,m}$ defined on $K^\times/q^\mathbb{Z}$ for $|q|_K<1$. For $\pi$ a uniformizer of $K$, set $v_K(\pi)=1$ (and so $v_K(p)=e$), let $|x|_K=p^{-fv_K(x)}$, and normalize the Haar measure $\mu^+_K$ by $\mu^+_K(\mathcal{O}_K)=1$ (where $\mathcal{O}_K$ is the ring of integers of $K$). Similar to with $\mathbb{Q}_p$, we need only consider the case of $q=\pi^m$ for some fixed uniformizer $\pi$. With the substitutions $\mathcal{O}_K$ for $\mathbb{Z}_p$, $\pi\in K$ for $p\in\mathbb{Q}_p$, and $p^f\in\mathbb{R}$ for $p\in\mathbb{R}$ (including taking logarithms to be base $p^f$ instead of base $p$), all of the computations for $B_{K,m}$ are identical to those for $B_{p,m}$. E.g., when \eqref{E:Uy} is translated to
\begin{equation}
    U_K(y)=p^{-fv_K(y)}+p^{-f(m-1-v_K(y))}=|y|_K+p^{-f(m-1)}|y|_K^{-1}\label{E:UKy}
\end{equation}
then in direct correspondence with \eqref{E:lambdan}, the coefficients $\lambda_n(y)$ for $n>0$ are given by
\begin{equation}
    \lambda_n(y)=U_K(y)^n\dfrac{p^{f(n+1)}-1}{(p^f-1)(p^f+1)^n(p^{fn}-1)}\lambda_0\label{E:lambdaKn}
\end{equation}
and in direct correspondence with \eqref{E:DBxyA}, whenever $x\neq y$,
\begin{equation}
    DB_{K,m}(x,y)=-\lambda_0\dfrac{p^{-f|v_K(x)-v_K(y)|}}{p^f-1}\dfrac{p^f+1}{p^f+1-U_K(y)}\label{E:DBKxyA}
\end{equation}
Similarly, we obtain in correspondence with \eqref{E:lambda0} that
\begin{equation}
    \lambda_0=\dfrac{p^f(p^f-1)}{p^f+1}\label{E:Klambda0}
\end{equation}
and in correspondence with \eqref{E:CRecur}, when we write $v_K(x)=q$ and $v_K(y)=r$, the symmetrized $C_{K,m}$ is given by
\begin{equation}
    C_{K,m}(q,r)=\begin{cases}
        C_{K,m}(m-1,0)-\dfrac{p^f}{m\Lambda_{m-2}} & q=m-2,r=0\\[10pt]
        \dfrac{\Lambda_q+\Lambda_{q+2}}{\Lambda_q}C_{K,m}(q+1,0)-\dfrac{\Lambda_{q+2}}{\Lambda_q}C_{K,m}(q+2,0)-\dfrac{p^f-1}{m\Lambda_q} & q<m-2,r=0\\[5pt]
        C_{K,m}(q,0)+C_{K,m}(m-1-r,0)-C_{K,m}(m-1,0) & q\geq r\\
        C_{K,m}(m-1-q,0)+C_{K,m}(r,0)-C_{K,m}(m-1,0) & q\leq r
    \end{cases}\label{E:CKRecur}
\end{equation}
where $C_{K,m}(m-1,0)$ is arbitrary and $\Lambda_q=1+p^{-f}-p^{-f}U_K(x)=1+p^{-f}-p^{-f(q+1)}-p^{-f(m-q)}$.

\newpage
\appendix
\section{Appendix: Computing \texorpdfstring{$Dd(x,y)^n$}{Dd(x,y)\^{}n} for \texorpdfstring{$n\geq1$}{n\textgreater=1}}

\subsection{Computing \texorpdfstring{$Dd(x,y)$}{Dd(x,y)}}

Here, we compute $Dd(x,y)$. We have
\begin{equation}
    Dd(x,y)=\int_{\bigcup_{s=0}^{m-1}p^s\mathbb{Z}_p^\times}\dfrac{d(z,y)-d(x,y)}{|z-x|^2}|x|\,dz\label{E:Ddxy}
\end{equation}
Again suppose $v_p(x)=q\neq r=v_p(y)$. As with $D\log(d(x,y))$, the integrand in \eqref{E:Ddxy} vanishes unless $v_p(z)=r$, and $\eqref{E:Ddxy}$ becomes
\begin{equation}
    Dd(x,y)=\int_{p^r\mathbb{Z}_p^\times}\dfrac{p^{r-v_p(z-y)}-1}{\max\{p^{-2q},p^{-2r}\}}p^{-q}\,dz
\end{equation}
Proceeding as with $D\log(d(x,y))$, we have that for $v_p(x)=q\neq r=v_p(y)$,
\begin{align}
    Dd(x,y)&=\int_{\bigcup_{i=r}^{\infty}\partial y_i}\dfrac{p^{r-v_p(z-y)}-1}{\max\{p^{-2q},p^{-2r}\}}p^{-q}\,dz\\
    &=\sum_{i=r}^{\infty}\int_{\partial y_i}\dfrac{p^{r-v_p(z-y)}-1}{\max\{p^{-2q},p^{-2r}\}}p^{-q}\,dz\nonumber\\
    &=\sum_{i=r}^{\infty}\int_{\partial y_i}\dfrac{p^{r-i}-1}{\max\{p^{-2q},p^{-2r}\}}p^{-q}\,dz\nonumber\\
    &=\sum_{i=r+1}^{\infty}\dfrac{p^{r-i}-1}{\max\{p^{-2q},p^{-2r}\}}p^{-q}(p-1)p^{-i-1}\label{E:Ddqr}
\end{align}
Evaluating \eqref{E:Ddqr}, we obtain that, when $v_p(x)=q\neq r=v_p(y)$,
\begin{equation}
    Dd(x,y)=-\dfrac{p^{-|q-r|}}{p+1}\label{E:DdqrA}
\end{equation}

Again suppose $v_p(x)=v_p(y)=r$ and $v_p(x-y)=\ell$. For $s\neq r$ we have
\begin{align}
    \int_{p^s\mathbb{Z}_p^\times}\dfrac{d(z,y)-d(x,y)}{|z-x|^2}|x|\,dz&=\int_{p^s\mathbb{Z}_p^\times}\dfrac{1-p^{r-\ell}}{\max\{p^{-2r},p^{-2s}\}}p^{-r}\,dz\nonumber\\
    &=\dfrac{1-p^{r-\ell}}{\max\{p^{-2r},p^{-2s}\}}(p-1)p^{-r-s-1}\label{E:Ddrrs}
\end{align}
By the same reasoning as with $D\log(d(x,y))$,
\begin{align}
    \int_{p^r\mathbb{Z}_p^\times}\dfrac{d(z,y)-d(x,y)}{|z-x|^2}|x|\,dz&=\sum_{i=r}^{\infty}\int_{\partial x_i}\dfrac{p^{r-v_p(z-y)}-p^{r-\ell}}{p^{-2i}}p^{-r}\,dz\nonumber\\
    &=\sum_{i=r}^{\ell-1}\int_{\partial x_i}\dfrac{p^{r-i}-p^{r-\ell}}{p^{-2i}}p^{-r}\,dz+\int_{\partial x_\ell}\dfrac{p^{r-v_p(z-y)}-p^{r-\ell}}{p^{-2\ell}}p^{-r}\,dz\nonumber\\
    &=\sum_{i=r}^{\ell-1}\int_{\partial x_i}\dfrac{p^{r-i}-p^{r-\ell}}{p^{-2i}}p^{-r}\,dz+\sum_{i=\ell+1}^{\infty}\int_{\partial y_i}\dfrac{p^{r-i}-p^{r-\ell}}{p^{-2\ell}}p^{-r}\,dz\nonumber\\
    &=\dfrac{p-2}{p}\left(1-p^{r-\ell}\right)+\sum_{i=r+1,i\neq\ell}^{\infty}\left(p^{r-i}-p^{r-\ell}\right)(p-1)p^{\min\{2i,2\ell\}-i-r-1}\nonumber\\
    &=-\dfrac{1}{p+1}-\dfrac{2}{p}\left(1-p^{-(\ell-r)}\right)+\dfrac{p-1}{p}(\ell-r)\label{E:Ddrrr}
\end{align}
Combining \eqref{E:Ddrrs} and \eqref{E:Ddrrr}, we obtain (for $v_p(x)=v_p(y)=r$ and $v_p(x-y)=\ell$)
\begin{align}
    &\int_{\bigcup_{s=0}^{m-1}p^s\mathbb{Z}_p^\times}\dfrac{d(z,y)-d(x,y)}{|z-x|^2}|x|\,dz\nonumber\\
    ={}&-\dfrac{1}{p+1}-\dfrac{2}{p}\left(1-p^{-(\ell-r)}\right)+\dfrac{p-1}{p}(\ell-r)+\sum_{s=0,s\neq r}^{m-1}\left(1-p^{r-\ell}\right)(p-1)p^{-|r-s|-1}\nonumber\\
    ={}&-\dfrac{1}{p+1}-\left(p^{r-m}+p^{-r-1}\right)\left(1-p^{-(\ell-r)}\right)+\dfrac{p-1}{p}(\ell-r)\label{E:DdrrA}
\end{align}

Taken together, \eqref{E:DdqrA} and \eqref{E:DdrrA} give
\begin{equation}
    Dd(x,y)=\begin{cases}
        -\dfrac{p^{-|v_p(x)-v_p(y)|}}{p+1} & v_p(x)\neq v_p(y)\\[15pt]
        \begin{aligned}
            -\dfrac{1}{p+1}-\dfrac{U(y)}{p}\left(1-p^{-(v_p(x-y)-v_p(y))}\right)\\
            +\dfrac{p-1}{p}(v_p(x-y)-v_p(y))
        \end{aligned} & v_p(x)=v_p(y)
    \end{cases}
\end{equation}

\subsection{Computing \texorpdfstring{$Dd(x,y)^n$}{Dd(x,y)\^{}n} for \texorpdfstring{$n>1$}{n\textgreater1}}

We now compute $Dd(x,y)^n$ for $n>1$. We have
\begin{equation}
    Dd(x,y)^n=\int_{\bigcup_{s=0}^{m-1}p^s\mathbb{Z}_p^\times}\dfrac{d(z,y)^n-d(x,y)^n}{|z-x|^2}|x|\,dz\label{E:Ddxyn}
\end{equation}
We again suppose $v_p(x)=q\neq r=v_p(y)$. As before, the integrand in \eqref{E:Ddxyn} vanishes unless $v_p(z)=r$, and $\eqref{E:Ddxyn}$ becomes
\begin{equation}
    Dd(x,y)^n=\int_{p^r\mathbb{Z}_p^\times}\dfrac{p^{n(r-v_p(z-y))}-1}{\max\{p^{-2q},p^{-2r}\}}p^{-q}\,dz
\end{equation}
Proceeding as before, for $v_p(x)=q\neq r=v_p(y)$ we have
\begin{align}
    Dd(x,y)^n&=\int_{\bigcup_{i=r}^{\infty}\partial y_i}\dfrac{p^{n(r-v_p(z-y))}-1}{\max\{p^{-2q},p^{-2r}\}}p^{-q}\,dz\\
    &=\sum_{i=r}^{\infty}\int_{\partial y_i}\dfrac{p^{n(r-v_p(z-y))}-1}{\max\{p^{-2q},p^{-2r}\}}p^{-q}\,dz\nonumber\\
    &=\sum_{i=r}^{\infty}\int_{\partial y_i}\dfrac{p^{n(r-i)}-1}{\max\{p^{-2q},p^{-2r}\}}p^{-q}\,dz\nonumber\\
    &=\sum_{i=r+1}^{\infty}\dfrac{p^{n(r-i)}-1}{\max\{p^{-2q},p^{-2r}\}}p^{-q}(p-1)p^{-i-1}\label{E:Ddnqr}
\end{align}
Evaluating \eqref{E:Ddnqr}, we obtain that, when $v_p(x)=q\neq r=v_p(y)$,
\begin{equation}
    Dd(x,y)^n=-\dfrac{(p^n-1)p^{-|q-r|}}{p^{n+1}-1}\label{E:DdnqrA}
\end{equation}

We again suppose $v_p(x)=v_p(y)=r$ and $v_p(x-y)=\ell$. For $s\neq r$ we have
\begin{align}
    \int_{p^s\mathbb{Z}_p^\times}\dfrac{d(z,y)^n-d(x,y)^n}{|z-x|^2}|x|\,dz&=\int_{p^s\mathbb{Z}_p^\times}\dfrac{1-p^{n(r-\ell)}}{\max\{p^{-2r},p^{-2s}\}}p^{-r}\,dz\nonumber\\
    &=\dfrac{1-p^{n(r-\ell)}}{\max\{p^{-2r},p^{-2s}\}}(p-1)p^{-r-s-1}\label{E:Ddnrrs}
\end{align}
By the same reasoning as before,
\begin{align}
    &\int_{p^r\mathbb{Z}_p^\times}\dfrac{d(z,y)^n-d(x,y)^n}{|z-x|^2}|x|\,dz\\
    ={}&\sum_{i=r}^{\infty}\int_{\partial x_i}\dfrac{p^{n(r-v_p(z-y))}-p^{n(r-\ell)}}{p^{-2i}}p^{-r}\,dz\nonumber\\
    ={}&\sum_{i=r}^{\ell-1}\int_{\partial x_i}\dfrac{p^{n(r-i)}-p^{n(r-\ell)}}{p^{-2i}}p^{-r}\,dz+\int_{\partial x_\ell}\dfrac{p^{n(r-v_p(z-y))}-p^{n(r-\ell)}}{p^{-2\ell}}p^{-r}\,dz\nonumber\\
    ={}&\sum_{i=r}^{\ell-1}\int_{\partial x_i}\dfrac{p^{n(r-i)}-p^{n(r-\ell)}}{p^{-2i}}p^{-r}\,dz+\sum_{i=\ell+1}^{\infty}\int_{\partial y_i}\dfrac{p^{n(r-i)}-p^{n(r-\ell)}}{p^{-2\ell}}p^{-r}\,dz\nonumber\\
    ={}&\dfrac{p-2}{p}\left(1-p^{n(r-\ell)}\right)+\sum_{i=r+1,i\neq\ell}^{\infty}\left(p^{n(r-i)}-p^{n(r-\ell)}\right)(p-1)p^{\min\{2i,2\ell\}-i-r-1}\nonumber\\
    ={}&\dfrac{p^{n-1}-p^{-1}}{p^{n-1}-1}-\dfrac{2}{p}\left(1-p^{-n(\ell-r)}\right)-p^{(1-n)(\ell-r)}\left(\dfrac{(p+1)(p^n-1)^2}{p(p^{n-1}-1)(p^{n+1}-1)}\right)\label{E:Ddnrrr}
\end{align}
Combining \eqref{E:Ddnrrs} and \eqref{E:Ddnrrr}, we obtain (for $v_p(x)=v_p(y)=r$ and $v_p(x-y)=\ell$)
\begin{align}
    &\int_{\bigcup_{s=0}^{m-1}p^s\mathbb{Z}_p^\times}\dfrac{d(z,y)-d(x,y)}{|z-x|^2}|x|\,dz\nonumber\\
    ={}&\dfrac{p^{n-1}-p^{-1}}{p^{n-1}-1}-\dfrac{2}{p}\left(1-p^{-n(\ell-r)}\right)-p^{(1-n)(\ell-r)}\left(\dfrac{(p+1)(p^n-1)^2}{p(p^{n-1}-1)(p^{n+1}-1)}\right)\nonumber\\
    &+\sum_{s=0,s\neq r}^{m-1}\left(1-p^{n(r-\ell)}\right)(p-1)p^{-|r-s|-1}\nonumber\\
    ={}&\dfrac{p^{n-1}-p^{-1}}{p^{n-1}-1}-\dfrac{U(y)}{p}\left(1-p^{-n(\ell-r)}\right)-p^{(1-n)(\ell-r)}\left(\dfrac{(p+1)(p^n-1)^2}{p(p^{n-1}-1)(p^{n+1}-1)}\right)\label{E:DdnrrA}
\end{align}

Taken together, \eqref{E:DdnqrA} and \eqref{E:DdnrrA} give
\begin{equation}
    Dd(x,y)^n=\begin{cases}
        -\dfrac{p^n-1}{p^{n+1}-1}p^{-|v_p(x)-v_p(y)|} & v_p(x)\neq v_p(y)\\[15pt]
        \begin{aligned}
            \dfrac{p^{n-1}-p^{-1}}{p^{n-1}-1}-\dfrac{U(y)}{p}\left(1-p^{-n(v_p(x-y)-v_p(y))}\right)\\
            -\dfrac{(p+1)(p^n-1)^2}{p(p^{n-1}-1)(p^{n+1}-1)}p^{(1-n)(v_p(x-y)-v_p(y))}
        \end{aligned} & v_p(x)=v_p(y)
    \end{cases}
\end{equation}

\newpage
\section{Appendix: Some numerical results for \texorpdfstring{$C_{p,m}(x,y)$}{C \textunderscore p,m(x,y)}}
\[
C_{p,2}=\left(\begin{array}{cc} -\frac{p^3}{2\,\left(p-1\right)\,\left(p+1\right)} & 0\\ 0 & -\frac{p^3}{2\,\left(p-1\right)\,\left(p+1\right)} \end{array}\right)
\]
\[
C_{p,3}=\left(\begin{array}{ccc} -\frac{p^4\,\left(p+1\right)}{\left(p-1\right)\,\left(p+2\right)\,\left(p^2+p+1\right)} & -\frac{p^3}{3\,\left(p-1\right)\,\left(p+2\right)} & 0\\ -\frac{p^3}{3\,\left(p-1\right)\,\left(p+2\right)} & -\frac{2\,p^3}{3\,\left(p-1\right)\,\left(p+2\right)} & -\frac{p^3}{3\,\left(p-1\right)\,\left(p+2\right)}\\ 0 & -\frac{p^3}{3\,\left(p-1\right)\,\left(p+2\right)} & -\frac{p^4\,\left(p+1\right)}{\left(p-1\right)\,\left(p+2\right)\,\left(p^2+p+1\right)} \end{array}\right)
\]
\[
C_{p,4}=\left(\begin{array}{cccc} -\frac{p^5\,\left(3\,p^2+2\,p+1\right)}{2\,\left(p^2+1\right)\,\left(p-1\right)\,{\left(p+1\right)}^3} & -\frac{p^4\,\left(3\,p+1\right)}{4\,\left(p-1\right)\,{\left(p+1\right)}^3} & -\frac{p^4}{4\,\left(p-1\right)\,{\left(p+1\right)}^2} & 0\\ -\frac{p^4\,\left(3\,p+1\right)}{4\,\left(p-1\right)\,{\left(p+1\right)}^3} & -\frac{p^4\,\left(2\,p+1\right)}{2\,\left(p-1\right)\,{\left(p+1\right)}^3} & -\frac{p^4}{2\,\left(p-1\right)\,{\left(p+1\right)}^2} & -\frac{p^4}{4\,\left(p-1\right)\,{\left(p+1\right)}^2}\\ -\frac{p^4}{4\,\left(p-1\right)\,{\left(p+1\right)}^2} & -\frac{p^4}{2\,\left(p-1\right)\,{\left(p+1\right)}^2} & -\frac{p^4\,\left(2\,p+1\right)}{2\,\left(p-1\right)\,{\left(p+1\right)}^3} & -\frac{p^4\,\left(3\,p+1\right)}{4\,\left(p-1\right)\,{\left(p+1\right)}^3}\\ 0 & -\frac{p^4}{4\,\left(p-1\right)\,{\left(p+1\right)}^2} & -\frac{p^4\,\left(3\,p+1\right)}{4\,\left(p-1\right)\,{\left(p+1\right)}^3} & -\frac{p^5\,\left(3\,p^2+2\,p+1\right)}{2\,\left(p^2+1\right)\,\left(p-1\right)\,{\left(p+1\right)}^3} \end{array}\right)
\]
\newgeometry{top=3cm, bottom=3cm, left=2cm, right=2cm}
\[
C_{p,5} = 
\left(
\begin{array}{cc}
-\frac{p^6(p+1)(2p^4+3p^3+3p^2+p+1)}{(p-1)(p^2+2p+2)(p^3+2p^2+p+1)(p^4+p^3+p^2+p+1)} & 
-\frac{p^5(2p+1)(3p+2)}{5(p-1)(p^2+2p+2)(p^3+2p^2+p+1)} \\
-\frac{p^5(2p+1)(3p+2)}{5(p-1)(p^2+2p+2)(p^3+2p^2+p+1)} & 
-\frac{p^5(7p^2+9p+4)}{5(p-1)(p^2+2p+2)(p^3+2p^2+p+1)} \\
-\frac{p^4(p+1)(3p^2+p+1)}{5(p-1)(p^2+2p+2)(p^3+2p^2+p+1)} & 
-\frac{p^4(2p+1)(2p^2+2p+1)}{5(p-1)(p^2+2p+2)(p^3+2p^2+p+1)} \\
-\frac{p^5}{5(p-1)(p^3+2p^2+p+1)} & 
-\frac{2p^5}{5(p-1)(p^3+2p^2+p+1)} \\
0 & 
-\frac{p^5}{5(p-1)(p^3+2p^2+p+1)}
\end{array}
\right.
\]
\[
\left.
\begin{array}{ccc}
-\frac{p^4(p+1)(3p^2+p+1)}{5(p-1)(p^2+2p+2)(p^3+2p^2+p+1)} & 
-\frac{p^5}{5(p-1)(p^3+2p^2+p+1)} & 
0 \\
-\frac{p^4(2p+1)(2p^2+2p+1)}{5(p-1)(p^2+2p+2)(p^3+2p^2+p+1)} & 
-\frac{2p^5}{5(p-1)(p^3+2p^2+p+1)} & 
-\frac{p^5}{5(p-1)(p^3+2p^2+p+1)} \\
-\frac{2p^4(p+1)(3p^2+p+1)}{5(p-1)(p^2+2p+2)(p^3+2p^2+p+1)} & 
-\frac{p^4(2p+1)(2p^2+2p+1)}{5(p-1)(p^2+2p+2)(p^3+2p^2+p+1)} & 
-\frac{p^4(p+1)(3p^2+p+1)}{5(p-1)(p^2+2p+2)(p^3+2p^2+p+1)} \\
-\frac{p^4(2p+1)(2p^2+2p+1)}{5(p-1)(p^2+2p+2)(p^3+2p^2+p+1)} & 
-\frac{p^5(7p^2+9p+4)}{5(p-1)(p^2+2p+2)(p^3+2p^2+p+1)} & 
-\frac{p^5(2p+1)(3p+2)}{5(p-1)(p^2+2p+2)(p^3+2p^2+p+1)} \\
-\frac{p^4(p+1)(3p^2+p+1)}{5(p-1)(p^2+2p+2)(p^3+2p^2+p+1)} & 
-\frac{p^5(2p+1)(3p+2)}{5(p-1)(p^2+2p+2)(p^3+2p^2+p+1)} & 
-\frac{p^6(p+1)(2p^4+3p^3+3p^2+p+1)}{(p-1)(p^2+2p+2)(p^3+2p^2+p+1)(p^4+p^3+p^2+p+1)}
\end{array}
\right)
\]
\[
C_{p,6} = 
\left(\begin{array}{cc}
-\frac{p^7(5p^5+2p^4+4p^3+2p^2+p+1)}{2(p-1)(p+1)(p^2+p+1)^2(p^2-p+1)(p^3+p^2+1)} & 
-\frac{p^6(2p+1)(5p^3+3p^2+3p+1)}{6(p-1)(p+1)(p^2+p+1)^2(p^3+p^2+1)} \\
-\frac{p^6(2p+1)(5p^3+3p^2+3p+1)}{6(p-1)(p+1)(p^2+p+1)^2(p^3+p^2+1)} & 
-\frac{p^6(11p^4+13p^3+12p^2+7p+2)}{6(p-1)(p+1)(p^2+p+1)^2(p^3+p^2+1)} \\
-\frac{p^5(6p^5+7p^4+5p^3+6p^2+2p+1)}{6(p-1)(p+1)(p^2+p+1)^2(p^3+p^2+1)} & 
-\frac{p^5(7p^4+2p^3+6p^2+2p+1)}{6(p-1)(p^2+p+1)^2(p^3+p^2+1)} \\
-\frac{p^5(3p^3+p^2+p+1)}{6(p-1)(p+1)(p^2+p+1)(p^3+p^2+1)} & 
-\frac{p^5(2p+1)(2p^2+1)}{6(p-1)(p+1)(p^2+p+1)(p^3+p^2+1)} \\
-\frac{p^6}{6(p-1)(p+1)(p^3+p^2+1)} & 
-\frac{p^6}{3(p-1)(p+1)(p^3+p^2+1)} \\
0 & 
-\frac{p^6}{6(p-1)(p+1)(p^3+p^2+1)}
\end{array}\right.
\]
\[
\left.\begin{array}{cc}
-\frac{p^5(6p^5+7p^4+5p^3+6p^2+2p+1)}{6(p-1)(p+1)(p^2+p+1)^2(p^3+p^2+1)} & 
-\frac{p^5(3p^3+p^2+p+1)}{6(p-1)(p+1)(p^2+p+1)(p^3+p^2+1)} \\
-\frac{p^5(7p^4+2p^3+6p^2+2p+1)}{6(p-1)(p^2+p+1)^2(p^3+p^2+1)} & 
-\frac{p^5(2p+1)(2p^2+1)}{6(p-1)(p+1)(p^2+p+1)(p^3+p^2+1)} \\
-\frac{p^5(9p^5+11p^4+10p^3+9p^2+4p+2)}{6(p-1)(p+1)(p^2+p+1)^2(p^3+p^2+1)} & 
-\frac{p^5(3p^3+p^2+p+1)}{3(p-1)(p+1)(p^2+p+1)(p^3+p^2+1)} \\
-\frac{p^5(3p^3+p^2+p+1)}{3(p-1)(p+1)(p^2+p+1)(p^3+p^2+1)} & 
-\frac{p^5(9p^5+11p^4+10p^3+9p^2+4p+2)}{6(p-1)(p+1)(p^2+p+1)^2(p^3+p^2+1)} \\
-\frac{p^5(2p+1)(2p^2+1)}{6(p-1)(p+1)(p^2+p+1)(p^3+p^2+1)} & 
-\frac{p^5(7p^4+2p^3+6p^2+2p+1)}{6(p-1)(p^2+p+1)^2(p^3+p^2+1)} \\
-\frac{p^5(3p^3+p^2+p+1)}{6(p-1)(p+1)(p^2+p+1)(p^3+p^2+1)} & 
-\frac{p^5(6p^5+7p^4+5p^3+6p^2+2p+1)}{6(p-1)(p+1)(p^2+p+1)^2(p^3+p^2+1)}
\end{array}\right.
\]
\[
\left.\begin{array}{cc}
-\frac{p^6}{6(p-1)(p+1)(p^3+p^2+1)} & 
0 \\
-\frac{p^6}{6(p-1)(p+1)(p^3+p^2+1)} & 
-\frac{p^6}{6(p-1)(p+1)(p^3+p^2+1)} \\
-\frac{p^5(2p+1)(2p^2+1)}{6(p-1)(p+1)(p^2+p+1)(p^3+p^2+1)} & 
-\frac{p^5(3p^3+p^2+p+1)}{6(p-1)(p+1)(p^2+p+1)(p^3+p^2+1)} \\
-\frac{p^5(7p^4+2p^3+6p^2+2p+1)}{6(p-1)(p^2+p+1)^2(p^3+p^2+1)} & 
-\frac{p^5(6p^5+7p^4+5p^3+6p^2+2p+1)}{6(p-1)(p+1)(p^2+p+1)^2(p^3+p^2+1)} \\
-\frac{p^6(11p^4+13p^3+12p^2+7p+2)}{6(p-1)(p+1)(p^2+p+1)^2(p^3+p^2+1)} & 
-\frac{p^6(2p+1)(5p^3+3p^2+3p+1)}{6(p-1)(p+1)(p^2+p+1)^2(p^3+p^2+1)} \\
-\frac{p^6(2p+1)(5p^3+3p^2+3p+1)}{6(p-1)(p+1)(p^2+p+1)^2(p^3+p^2+1)} & 
-\frac{p^7(5p^5+2p^4+4p^3+2p^2+p+1)}{2(p-1)(p+1)(p^2+p+1)^2(p^2-p+1)(p^3+p^2+1)}
\end{array}\right)
\]
\newpage
\[
C_{p,7} = 
\left(\begin{array}{c}
-\frac{p^8(p+1)(3p^{11}+10p^{10}+17p^9+23p^8+24p^7+22p^6+18p^5+14p^4+8p^3+5p^2+2p+1)}{(p-1)(p^3+2p^2+2p+2)(p^4+2p^3+2p^2+p+1)(p^5+2p^4+p^3+p^2+p+1)(p^6+p^5+p^4+p^3+p^2+p+1)} \\
-\frac{p^7(15p^7+46p^6+57p^5+53p^4+39p^3+22p^2+11p+2)}{7(p-1)(p^3+2p^2+2p+2)(p^4+2p^3+2p^2+p+1)(p^5+2p^4+p^3+p^2+p+1)} \\
-\frac{p^6(p+1)(10p^7+21p^6+16p^5+18p^4+14p^3+13p^2+4p+2)}{7(p-1)(p^3+2p^2+2p+2)(p^4+2p^3+2p^2+p+1)(p^5+2p^4+p^3+p^2+p+1)} \\
-\frac{p^5(6p^9+19p^8+25p^7+26p^6+25p^5+21p^4+14p^3+7p^2+3p+1)}{7(p-1)(p^3+2p^2+2p+2)(p^4+2p^3+2p^2+p+1)(p^5+2p^4+p^3+p^2+p+1)} \\
-\frac{p^6(p+1)(3p^4+p^3+p^2+p+1)}{7(p-1)(p^4+2p^3+2p^2+p+1)(p^5+2p^4+p^3+p^2+p+1)} \\
-\frac{p^7}{7(p-1)(p^5+2p^4+p^3+p^2+p+1)} \\
0
\end{array}\right.
\]
\[
\left.\begin{array}{cc}
-\frac{p^7(15p^7+46p^6+57p^5+53p^4+39p^3+22p^2+11p+2)}{7(p-1)(p^3+2p^2+2p+2)(p^4+2p^3+2p^2+p+1)(p^5+2p^4+p^3+p^2+p+1)} & 
-\frac{p^6(p+1)(10p^7+21p^6+16p^5+18p^4+14p^3+13p^2+4p+2)}{7(p-1)(p^3+2p^2+2p+2)(p^4+2p^3+2p^2+p+1)(p^5+2p^4+p^3+p^2+p+1)} \\
-\frac{p^7(16p^7+50p^6+65p^5+64p^4+50p^3+30p^2+15p+4)}{7(p-1)(p^3+2p^2+2p+2)(p^4+2p^3+2p^2+p+1)(p^5+2p^4+p^3+p^2+p+1)} & 
-\frac{p^6(11p^8+35p^7+45p^6+45p^5+43p^4+35p^3+21p^2+8p+2)}{7(p-1)(p^3+2p^2+2p+2)(p^4+2p^3+2p^2+p+1)(p^5+2p^4+p^3+p^2+p+1)} \\
-\frac{p^6(11p^8+35p^7+45p^6+45p^5+43p^4+35p^3+21p^2+8p+2)}{7(p-1)(p^3+2p^2+2p+2)(p^4+2p^3+2p^2+p+1)(p^5+2p^4+p^3+p^2+p+1)} & 
-\frac{p^6(p+1)(13p^7+28p^6+25p^5+29p^4+21p^3+19p^2+8p+4)}{7(p-1)(p^3+2p^2+2p+2)(p^4+2p^3+2p^2+p+1)(p^5+2p^4+p^3+p^2+p+1)} \\
-\frac{p^5(7p^9+23p^8+33p^7+37p^6+36p^5+29p^4+18p^3+9p^2+3p+1)}{7(p-1)(p^3+2p^2+2p+2)(p^4+2p^3+2p^2+p+1)(p^5+2p^4+p^3+p^2+p+1)} & 
-\frac{p^5(9p^9+29p^8+41p^7+46p^6+43p^5+34p^4+24p^3+13p^2+5p+1)}{7(p-1)(p^3+2p^2+2p+2)(p^4+2p^3+2p^2+p+1)(p^5+2p^4+p^3+p^2+p+1)} \\
-\frac{p^6(2p+1)(2p^4+2p^3+p^2+p+1)}{7(p-1)(p^4+2p^3+2p^2+p+1)(p^5+2p^4+p^3+p^2+p+1)} & 
-\frac{2p^6(p+1)(3p^4+p^3+p^2+p+1)}{7(p-1)(p^4+2p^3+2p^2+p+1)(p^5+2p^4+p^3+p^2+p+1)} \\
-\frac{2p^7}{7(p-1)(p^5+2p^4+p^3+p^2+p+1)} & 
-\frac{p^6(2p+1)(2p^4+2p^3+p^2+p+1)}{7(p-1)(p^4+2p^3+2p^2+p+1)(p^5+2p^4+p^3+p^2+p+1)} \\
-\frac{p^7}{7(p-1)(p^5+2p^4+p^3+p^2+p+1)} & 
-\frac{p^6(p+1)(3p^4+p^3+p^2+p+1)}{7(p-1)(p^4+2p^3+2p^2+p+1)(p^5+2p^4+p^3+p^2+p+1)}
\end{array}\right.
\]
\[
\left.\begin{array}{cc}
-\frac{p^5(6p^9+19p^8+25p^7+26p^6+25p^5+21p^4+14p^3+7p^2+3p+1)}{7(p-1)(p^3+2p^2+2p+2)(p^4+2p^3+2p^2+p+1)(p^5+2p^4+p^3+p^2+p+1)} & 
-\frac{p^6(p+1)(3p^4+p^3+p^2+p+1)}{7(p-1)(p^4+2p^3+2p^2+p+1)(p^5+2p^4+p^3+p^2+p+1)} \\
-\frac{p^5(7p^9+23p^8+33p^7+37p^6+36p^5+29p^4+18p^3+9p^2+3p+1)}{7(p-1)(p^3+2p^2+2p+2)(p^4+2p^3+2p^2+p+1)(p^5+2p^4+p^3+p^2+p+1)} & 
-\frac{p^6(2p+1)(2p^4+2p^3+p^2+p+1)}{7(p-1)(p^4+2p^3+2p^2+p+1)(p^5+2p^4+p^3+p^2+p+1)} \\
-\frac{p^5(9p^9+29p^8+41p^7+46p^6+43p^5+34p^4+24p^3+13p^2+5p+1)}{7(p-1)(p^3+2p^2+2p+2)(p^4+2p^3+2p^2+p+1)(p^5+2p^4+p^3+p^2+p+1)} & 
-\frac{2p^6(p+1)(3p^4+p^3+p^2+p+1)}{7(p-1)(p^4+2p^3+2p^2+p+1)(p^5+2p^4+p^3+p^2+p+1)} \\
-\frac{2p^5(6p^9+19p^8+25p^7+26p^6+25p^5+21p^4+14p^3+7p^2+3p+1)}{7(p-1)(p^3+2p^2+2p+2)(p^4+2p^3+2p^2+p+1)(p^5+2p^4+p^3+p^2+p+1)} & 
-\frac{p^5(9p^9+29p^8+41p^7+46p^6+43p^5+34p^4+24p^3+13p^2+5p+1)}{7(p-1)(p^3+2p^2+2p+2)(p^4+2p^3+2p^2+p+1)(p^5+2p^4+p^3+p^2+p+1)} \\
-\frac{p^5(9p^9+29p^8+41p^7+46p^6+43p^5+34p^4+24p^3+13p^2+5p+1)}{7(p-1)(p^3+2p^2+2p+2)(p^4+2p^3+2p^2+p+1)(p^5+2p^4+p^3+p^2+p+1)} & 
-\frac{p^6(p+1)(13p^7+28p^6+25p^5+29p^4+21p^3+19p^2+8p+4)}{7(p-1)(p^3+2p^2+2p+2)(p^4+2p^3+2p^2+p+1)(p^5+2p^4+p^3+p^2+p+1)} \\
-\frac{p^5(7p^9+23p^8+33p^7+37p^6+36p^5+29p^4+18p^3+9p^2+3p+1)}{7(p-1)(p^3+2p^2+2p+2)(p^4+2p^3+2p^2+p+1)(p^5+2p^4+p^3+p^2+p+1)} & 
-\frac{p^6(11p^8+35p^7+45p^6+45p^5+43p^4+35p^3+21p^2+8p+2)}{7(p-1)(p^3+2p^2+2p+2)(p^4+2p^3+2p^2+p+1)(p^5+2p^4+p^3+p^2+p+1)} \\
-\frac{p^5(6p^9+19p^8+25p^7+26p^6+25p^5+21p^4+14p^3+7p^2+3p+1)}{7(p-1)(p^3+2p^2+2p+2)(p^4+2p^3+2p^2+p+1)(p^5+2p^4+p^3+p^2+p+1)} & 
-\frac{p^6(p+1)(10p^7+21p^6+16p^5+18p^4+14p^3+13p^2+4p+2)}{7(p-1)(p^3+2p^2+2p+2)(p^4+2p^3+2p^2+p+1)(p^5+2p^4+p^3+p^2+p+1)}
\end{array}\right.
\]
\[
\left.\begin{array}{cc}
 -\frac{p^7}{7\,\left(p-1\right)\,\left(p^5+2\,p^4+p^3+p^2+p+1\right)} \\ 
 -\frac{2\,p^7}{7\,\left(p-1\right)\,\left(p^5+2\,p^4+p^3+p^2+p+1\right)} \\ 
 -\frac{p^6\,\left(2\,p+1\right)\,\left(2\,p^4+2\,p^3+p^2+p+1\right)}{7\,\left(p-1\right)\,\left(p^4+2\,p^3+2\,p^2+p+1\right)\,\left(p^5+2\,p^4+p^3+p^2+p+1\right)} \\ 
 -\frac{p^5\,\left(7\,p^9+23\,p^8+33\,p^7+37\,p^6+36\,p^5+29\,p^4+18\,p^3+9\,p^2+3\,p+1\right)}{7\,\left(p-1\right)\,\left(p^3+2\,p^2+2\,p+2\right)\,\left(p^4+2\,p^3+2\,p^2+p+1\right)\,\left(p^5+2\,p^4+p^3+p^2+p+1\right)} \\
 -\frac{p^6\,\left(11\,p^8+35\,p^7+45\,p^6+45\,p^5+43\,p^4+35\,p^3+21\,p^2+8\,p+2\right)}{7\,\left(p-1\right)\,\left(p^3+2\,p^2+2\,p+2\right)\,\left(p^4+2\,p^3+2\,p^2+p+1\right)\,\left(p^5+2\,p^4+p^3+p^2+p+1\right)} \\
 -\frac{p^7\,\left(16\,p^7+50\,p^6+65\,p^5+64\,p^4+50\,p^3+30\,p^2+15\,p+4\right)}{7\,\left(p-1\right)\,\left(p^3+2\,p^2+2\,p+2\right)\,\left(p^4+2\,p^3+2\,p^2+p+1\right)\,\left(p^5+2\,p^4+p^3+p^2+p+1\right)} \\
 -\frac{p^7\,\left(15\,p^7+46\,p^6+57\,p^5+53\,p^4+39\,p^3+22\,p^2+11\,p+2\right)}{7\,\left(p-1\right)\,\left(p^3+2\,p^2+2\,p+2\right)\,\left(p^4+2\,p^3+2\,p^2+p+1\right)\,\left(p^5+2\,p^4+p^3+p^2+p+1\right)}
\end{array}\right.
\]
\[
\left.\begin{array}{cc}
0\\ 
-\frac{p^7}{7\,\left(p-1\right)\,\left(p^5+2\,p^4+p^3+p^2+p+1\right)}\\ 
-\frac{p^6\,\left(p+1\right)\,\left(3\,p^4+p^3+p^2+p+1\right)}{7\,\left(p-1\right)\,\left(p^4+2\,p^3+2\,p^2+p+1\right)\,\left(p^5+2\,p^4+p^3+p^2+p+1\right)}\\ 
-\frac{p^5\,\left(6\,p^9+19\,p^8+25\,p^7+26\,p^6+25\,p^5+21\,p^4+14\,p^3+7\,p^2+3\,p+1\right)}{7\,\left(p-1\right)\,\left(p^3+2\,p^2+2\,p+2\right)\,\left(p^4+2\,p^3+2\,p^2+p+1\right)\,\left(p^5+2\,p^4+p^3+p^2+p+1\right)}\\ 
-\frac{p^6\,\left(p+1\right)\,\left(10\,p^7+21\,p^6+16\,p^5+18\,p^4+14\,p^3+13\,p^2+4\,p+2\right)}{7\,\left(p-1\right)\,\left(p^3+2\,p^2+2\,p+2\right)\,\left(p^4+2\,p^3+2\,p^2+p+1\right)\,\left(p^5+2\,p^4+p^3+p^2+p+1\right)}\\ 
-\frac{p^7\,\left(15\,p^7+46\,p^6+57\,p^5+53\,p^4+39\,p^3+22\,p^2+11\,p+2\right)}{7\,\left(p-1\right)\,\left(p^3+2\,p^2+2\,p+2\right)\,\left(p^4+2\,p^3+2\,p^2+p+1\right)\,\left(p^5+2\,p^4+p^3+p^2+p+1\right)}\\ 
-\frac{p^8\,\left(p+1\right)\,\left(3\,p^{11}+10\,p^{10}+17\,p^9+23\,p^8+24\,p^7+22\,p^6+18\,p^5+14\,p^4+8\,p^3+5\,p^2+2\,p+1\right)}{\left(p-1\right)\,\left(p^3+2\,p^2+2\,p+2\right)\,\left(p^4+2\,p^3+2\,p^2+p+1\right)\,\left(p^5+2\,p^4+p^3+p^2+p+1\right)\,\left(p^6+p^5+p^4+p^3+p^2+p+1\right)}
\end{array}\right)
\]
\restoregeometry

\bibliographystyle{unsrt}
\bibliography{references}

@article{huang2021greens,
  author = {Huang, An and Stoica, Bogdan and Yau, Shing-Tung and Zhong, Xiao},
  title = {{Green's} functions for {Vladimirov} derivatives and {Tate's} thesis},
  journal = {Communications in Number Theory and Physics},
  volume = {15},
  number = {2},
  pages = {315--361},
  year = {2021},
  issn = {1931-4523, 1931-4531},
  doi = {10.4310/CNTP.2021.v15.n2.a3},
  url = {https://doi.org/10.4310/CNTP.2021.v15.n2.a3},
  mrnumber = {MR4278926}
}

@article{lin2010elliptic,
  author = {Lin, Chang-Shou and Wang, Chin-Lung},
  title = {Elliptic functions, {Green} functions and the mean field equations on tori},
  journal = {Annals of Mathematics},
  volume = {172},
  number = {2},
  pages = {911--954},
  year = {2010},
  issn = {0003-486X, 1939-8980},
  doi = {10.4007/annals.2010.172.911},
  url = {https://doi.org/10.4007/annals.2010.172.911},
  mrnumber = {MR2680484}
}

@incollection{lin2010function,
  author = {Lin, Chang-Shou and Wang, Chin-Lung},
  title = {A function theoretic view of the mean field equations on tori},
  booktitle = {Recent Advances in Geometric Analysis},
  series = {Advanced Lectures in Mathematics (ALM)},
  volume = {11},
  pages = {173--193},
  year = {2010},
  isbn = {978-1-57146-143-8},
  mrnumber = {MR2648944}
}

@article{FO1987,
  author = {P. G. O. Freund and M. Olson},
  title = {{Non-Archimedean} strings},
  journal = {Physics Letters B},
  series = {Advanced Lectures in Mathematics (ALM)},
  volume = {199},
  number = {2},
  pages = {186-190},
  year = {1987}
  
}

@book{Silverman,
  author = {J. H. Silverman},
  title = {Advanced Topics in the Arithmetic of Elliptic Curves},
  
  series = {Graduate Texts in Mathematics},
  volume = {151},
  
  year = {1994}
  
}

@incollection{HJ2024,
  author = {An Huang and Christian Baadsgaard Jepsen},
  title = {Finite temperature at finite places},
  booktitle = {Journal of High Energy Physics},
  
  volume = {97},
  
  year = {2025}
  
}

@article{FW1987,
  author = {P. G. O. Freund and E. Witten},
  title = {Adelic string amplitudes},
  journal = {Physics Letters B},
  volume = {199},
  number = {2},
  year = {1987},
  pages = {191-194},
  
}

@incollection{Z1989,
  author = {A. V. Zabrodin},
  title = {Non-{A}rchimedean Strings and {B}ruhat-{T}its Trees},
  booktitle = {Communications in Mathematical
Physics},
  
  volume = {123, 3},
  
  year = {1989}
  
}

@incollection{HSZ2022,
  author = {An Huang and Bogdan Stoica and Xiao Zhong},
  title = {Quadratic reciprocity from a family of adelic conformal field theories},
  booktitle = {arXiv:2202.01217},
  year ={2022}

}

@incollection{CL2007,
  author = {Luis Caffarelli and Luis Silvestre},
  title = {An extension problem related to the fractional {Laplacian}},
  booktitle = {Comm. in Partial Differential Equations},
  
  volume = {32},
  pages = {1245--1260},
  year = {2007}
  
}

@incollection{HMST,
  author = {Matthew Heydeman and Matilde Marcolli and
Ingmar A. Saberi and Bogdan Stoica},
  title = {Tensor networks, $p$-adic fields, and
algebraic curves: arithmetic and
the $\text{AdS}_3$/$\text{CFT}_2$ correspondence},
  booktitle = {ADV. THEOR. MATH. PHYS},
  
  volume = {22},
  pages = {93--176},
  year = {2018}
  
}

@incollection{Bootstrap,
  author = {Luca Iliesiu and Murat Kologlu and Raghu Mahajan and Eric Perlmutter and David Simmons-Duffin},
  title = {The conformal bootstrap at finite temperature},
  booktitle = {Journal of High Energy Physics},
  
  volume = {2018},

  year = {2018}
  
}

@incollection{HJ2025,
  author = {An Huang and Christian Baadsgaard Jepsen},
  title = {A glimpse into the Ultrametric spectrum},
  booktitle = {arXiv:2601.03738},
  
}

@incollection{PS,
  author = {Michael E. Peskin and Daniel V. Schroeder},
  title = {An Introduction to Quantum Field Theory},
  booktitle = {Frontiers in Physics},
  
}

@incollection{bradley2025boundaryvalueproblemspadic,
      title={Boundary Value Problems for p-Adic Elliptic {Parisi-Z\'u\~niga} Diffusion}, 
      author={Patrick Erik Bradley},
      booktitle = {arXiv:2504.06288},
      year ={2025}

      
}

@incollection{hassan2025padichighergreensfunctions,
      title={$p$-adic Higher {Green's} Functions for {Stark-Heegner} Cycles}, 
      author={Hazem Hassan},
      booktitle = {arXiv:2509.09446},
      year ={2025}
}

@incollection{bradley2025diffusionoperatorspadicanalytic,
      title={Diffusion operators on $p$-adic analytic manifolds}, 
      author={Patrick Erik Bradley},
      booktitle = {arXiv:2510.22563},
      year ={2025}
}

@book {MR3586737,
    AUTHOR = {Z\'u\~niga-Galindo, W. A.},
     TITLE = {Pseudodifferential equations over non-{A}rchimedean spaces},
    SERIES = {Lecture Notes in Mathematics},
    VOLUME = {2174},
 PUBLISHER = {Springer, Cham},
      YEAR = {2016},
     PAGES = {xvi+175},
      ISBN = {978-3-319-46737-5; 978-3-319-46738-2},
   MRCLASS = {35-02 (11S80 14G10 35S05 46S20 47G30 58J40 60J25)},
  MRNUMBER = {3586737},
MRREVIEWER = {Anatoly\ N.\ Kochubei},
       DOI = {10.1007/978-3-319-46738-2},
       URL = {https://doi.org/10.1007/978-3-319-46738-2},
}

@article {MR59482,
    AUTHOR = {Kac, M. and Murdock, W. L. and Szeg\"o, G.},
     TITLE = {On the eigenvalues of certain {H}ermitian forms},
   JOURNAL = {J. Rational Mech. Anal.},
  FJOURNAL = {Journal of Rational Mechanics and Analysis},
    VOLUME = {2},
      YEAR = {1953},
     PAGES = {767--800},
      ISSN = {1943-5282,1943-5290},
   MRCLASS = {46.3X},
  MRNUMBER = {59482},
MRREVIEWER = {F.\ Smithies},
       DOI = {10.1512/iumj.1953.2.52034},
       URL = {https://doi.org/10.1512/iumj.1953.2.52034},
}

\noindent An Huang, anhuang@brandeis.edu\\
\emph{Department of Mathematics, Brandeis University, Waltham, MA 02453, USA}\\[-8pt]

\noindent Rebecca Rohrlich, rebeccarohrlich@brandeis.edu\\
\emph{Department of Mathematics, Brandeis University, Waltham, MA 02453, USA}\\[-8pt]

\noindent Yaojia Sun, 24210180117@m.fudan.edu.cn\\
\emph{School of Mathematical Sciences, Fudan University, Shanghai, 200433, P.R. China}\\[-8pt]

\noindent Eric Whyman, ewhyman@brandeis.edu\\
\emph{Department of Mathematics, Brandeis University, Waltham, MA 02453, USA}\\

\end{document}